\documentclass[11pt, reqno]{amsart}
\usepackage[utf8]{inputenc}
\usepackage{amssymb}
\usepackage{mathtools}
\usepackage[hyperfootnotes=false]{hyperref}
\usepackage{cleveref}
\usepackage[multiple]{footmisc}
\usepackage[a4paper, total={6.3in, 9.5in}, centering]{geometry}

\newcommand{\N}{\mathbb{N}}
\newcommand{\Z}{\mathbb{Z}}

\newcommand{\R}{\mathbb{R}}

\newcommand{\lcm}{\mathrm{lcm}}
\newcommand{\ord}{\mathrm{ord}}
\newcommand{\wt}{\widetilde}

\renewcommand{\P}{\mathbb{P}}

\DeclareMathOperator*{\E}{\mathbb{E}}
\allowdisplaybreaks

\theoremstyle{plain}
\newtheorem{theorem}{Theorem}
\newtheorem{lemma}[theorem]{Lemma}
\newtheorem{proposition}[theorem]{Proposition}
\newtheorem{corollary}[theorem]{Corollary}

\newtheorem{fact}[theorem]{Fact}

\theoremstyle{definition}

\newtheorem{remark}[theorem]{Remark}

\numberwithin{theorem}{section}

\begin{document}

\title{The Rényi entropy of the order of a random permutation}
\author{Adrian Beker}
\address{Department of Mathematics, Faculty of Science, University of Zagreb, Bijenička cesta 30, 10000 Zagreb, Croatia}
\email{adrian.beker@math.hr}

\begin{abstract}
    We study the distribution of the order of a random permutation of $[n]$ through the lens of Rényi entropy. In particular, we obtain an asymptotic for the Rényi $q$-entropy of the order in the full range $1 \leq q \leq \infty$. For $q > 1$, our results are quantitatively optimal and reveal a tight connection between the asymptotic behaviour of the Rényi $q$-entropy and arithmetic properties of $n$. Of particular interest are the cases $q = \infty$ and $q = 2$, which correspond to the maximum probability of achieving a particular order and the probability that two independent random permutations have equal orders, respectively. In the former case, we show that the probability in question is asymptotic to $1/n$ and additionally characterise the maximiser for sufficiently large $n$. In the latter case, we determine a minimal and maximal order for the probability as a function of $n$, of respective forms $c/n^2$ and $\log^*n/n^2$. Our results provide an essentially complete answer to a set of questions raised by Acan, Burnette, Eberhard, Schmutz and Thomas, some of which go back to work of Erdős and Turán from the 1960s.
\end{abstract}

\maketitle

\section{Introduction}\label{sec:intro}

\subsection{History and overview}\label{subsec:history}

Let $\pi_n$ be a permutation chosen uniformly at random from $S_n$, the symmetric group on $n$ letters. Let $\ord(\pi_n)$ denote the order of $\pi_n$, i.e.\ the least common multiple of the lengths of its cycles. Understanding the distribution $\ord(\pi_n)$ is a fundamental problem in probabilistic group theory. Its study has a rich history, which goes back more than a hundred years to the work of Landau \cite{landau}. In this work, Landau showed that the maximum value of the order is $e^{(1+o(1))\sqrt{n\log n}}$. Later on, a systematic investigation of the statistical properties of $\ord(\pi_n)$ was undertaken by Erdős and Turán. In a series of works \cite{erdos-turan-groups-i, erdos-turan-groups-ii, erdos-turan-groups-iii, erdos-turan-groups-iv}, they established a number of results concerning the distribution of $\ord(\pi_n)$. We highlight the celebrated Erdős--Turán law \cite{erdos-turan-groups-iii}, which states that $\log\ord(\pi_n)$ is asymptotically normally distributed with mean $\frac{1}{2}\log^2n$ and variance $\frac{1}{3}\log^3n$. The study of orders of random permutations and their cycle type has since led to fruitful developments at the interface of probability and number theory. For example, on the probabilistic side, there has been interest in rates of convergence/large deviation estimates for the Erdős--Turán law, as well as extensions to non-uniform distributions \cite{nicolas, barbour-tavare, storm-zeindler}. On the number-theoretic side, there are striking analogies between the cycle structure of a random permutation and the prime factorisation of a typical integer. This is the subject of the \emph{anatomy of integers and permutations}; the interested reader is invited to consult \cite{granville, ford-anatomy}. For a more complete account of the literature in the area, we recommend \cite[\S1]{ford}.

While the macroscopic behaviour of $\ord(\pi_n)$ is by now fairly well understood, obtaining local limit results has proved more challenging. This is primarily due to the fact that point probabilities 
\[
    p_n(m) \vcentcolon= \P(\ord(\pi_n) = m)
\]
depend sensitively on arithmetic properties of $m$ and $n$. Indeed, permutations $\pi \in S_n$ of order $m$ correspond via their cycle type to representations of $n$ as a sum of divisors of $m$. If $m$ is fixed, the problem of counting objects of this kind can be successfully addressed using methods of analytic combinatorics (see \cite{wilf}). However, when $m$ is allowed to grow with $n$, the interaction between additive and multiplicative structure makes it difficult to estimate $p_n(m)$. In this direction, Acan, Burnette, Eberhard, Schmutz and Thomas \cite{acan-burnette-eberhard-schmutz-thomas} recently studied the collision probability for $\ord(\pi_n)$. Letting $\pi_n'$ be an independent copy of $\pi_n$, they were interested in estimating the probability that $\pi_n$ and $\pi_n'$ have equal orders:
\[
    P_2(n) \vcentcolon= \P(\ord(\pi_n) = \ord(\pi_n')).
\]
Note that this probability can be expressed in terms of the $\ell^2$-norm of $p_n$:
\[
    P_2(n) = \sum_{m\in\N}p_n(m)^2 = \lVert p_n\rVert_2^2.
\]
The quantity $P_2(n)$ is closely related to the Rényi $2$-entropy, also known as the \emph{collision entropy}, of $\ord(\pi_n)$. Recall that, for a random variable $Y$ taking finitely many values and a parameter $q \in (1,\infty)$, the \emph{Rényi $q$-entropy} of $Y$ is defined to be
\[
    H_q(Y) \vcentcolon= \frac{1}{1-q}\log\Bigl(\sum_{y}\P(Y=y)^q\Bigr) = \frac{q}{1-q}\log(\lVert p_Y\rVert_q),
\]
where $p_Y(y) \vcentcolon= \P(Y = y)$ is the probability mass function of $Y$. For $q \in \{1,\infty\}$, the Rényi entropy is defined by
\[
    H_q(Y) \vcentcolon= \lim_{q'\to q}H_{q'}(Y).
\]
$H_{\infty}$ is also called \emph{min-entropy} and equals
\[
    H_{\infty}(Y) = \min_y\log\Bigl(\frac{1}{p_Y(y)}\Bigr) = \log\Bigl(\frac{1}{\lVert p_Y\rVert_{\infty}}\Bigr).
\]
At the other extreme, $H_1$ is the \emph{Shannon entropy}, often referred to simply as \emph{entropy}, and can be expressed as
\[
    H_1(Y) = \sum_{y}p_Y(y)\log\Bigl(\frac{1}{p_Y(y)}\Bigr).
\]
Rényi entropy is an information-theoretic concept which provides a way of measuring the randomness or uncertainty of a system; the parameter $q$ determines how much attention is given to rare events versus common ones. In our context, it will capture certain features of the distribution of $\ord(\pi_n)$ which are not visible from limit laws of Erdős--Turán type.

\subsection{Previous work and open problems}\label{subsec:prior_work}

The aforementioned work \cite{acan-burnette-eberhard-schmutz-thomas} (see also \cite{eberhard}) showed that, contrary to a conjecture of Godin \cite{godin}, $P_2(n)$ is not $O(1/n^2)$. Their argument was based on the following construction. Given a positive integer $n$, consider the set
\[
    K_n \vcentcolon= \bigl\{k \in \{0,1,\ldots,n-1\}\ \mid\ \lcm(1,\ldots,k) \mid n-k\bigr\}.
\]
By design, whenever $\pi_n$ has a cycle of length $n-k$ with $k \in K_n$, $\pi_n$ has order exactly $n-k$. If $n$ is large, $\pi_n$ can contain at most one such cycle, and this happens with probability $1/(n-k)$. One can show that the size of $K_n$ is unbounded as $n$ varies; the argument is then completed by considering the contribution of orders of the form $n-k$ with $k \in K_n$.

Notwithstanding the above result, Acan et al.\ \cite{acan-burnette-eberhard-schmutz-thomas} established an upper bound of the form
\begin{equation}\label{eq:equal_order_prob}
    P_2(n) \leq n^{-2+o(1)}.
\end{equation}
It is natural to wonder whether one can prove more precise estimates for $P_2(n)$, and Acan et al.\ \cite[\S6]{acan-burnette-eberhard-schmutz-thomas} raised several questions in this direction. First, they noticed that the values of $n$ for which $P_2(n)$ is large seem to have a particular arithmetic form, and asked what happens for $n$ not of this form, e.g.\ $n = k!+1$. This was taken further by Eberhard in the blog post \cite{eberhard-blog}, where it was conjectured that the limit inferior of $n^2P_2(n)$ is finite, that $n^2P_2(n)$ in fact converges along subsequences of the above form, and that $n^2P_2(n)$ is bounded by a very slowly growing function of $n$.

In a related direction, Acan et al.\ \cite[\S6]{acan-burnette-eberhard-schmutz-thomas} also discussed the maximum probability that $\ord(\pi_n)$ equals a particular value. Denoting this quantity by
\[
    P_{\infty}(n) \vcentcolon =\lVert p_n\rVert_{\infty} = \max_{m\in \N}p_n(m),
\]
they noted that
\begin{equation}\label{eq:lower_bound_cycle}
    P_{\infty}(n) \geq \P(\pi_n \text{ is an $n$-cycle}) = 1/n,
\end{equation}
and observed that $P_{\infty}(n)$ can be slightly larger than $1/n$. In the other direction, since $\lVert p_n\rVert_{\infty} \leq \lVert p_n\rVert_2$, their results on the collision probability readily imply an upper bound of the form
\begin{equation*}
    P_{\infty}(n) \leq n^{-1+o(1)}.
\end{equation*}
Acan et al.\ therefore raised the natural question of obtaining more precise bounds for the maximum of $p_n(m)$ and determining the values of $m$ for which this maximum is attained. They attributed this question to Erdős and Turán, who originally posed it in their survey \cite{erdos-turan-groups}, and reiterated it in their later paper \cite{erdos-turan-groups-iv}. This problem also appears in the booklet \emph{Some of Paul's favorite problems}, produced by a group of authors for the conference ``Paul Erdős and his mathematics'', held in Budapest in 1999. As such, it was recently indexed by Thomas Bloom on the website \emph{Erdős Problems} \cite{bloom-erdosproblems}.

\subsection{Main results}\label{subsec:main_results}

In this paper, we make substantial progress on the above problems by giving an essentially complete answer to the questions raised by Acan et al. In fact, our results are more general and address the Rényi $q$-entropy in the full range of the parameter $q$. Given $q \in (1,\infty)$, we define
\[
    P_q(n) \vcentcolon= \lVert p_n\rVert_q^q = \sum_{m\in\N}p_n(m)^q.
\]
This quantity can be thought of as a higher-order variant of the collision probability for $\ord(\pi_n)$: if $q$ is an integer, then $P_q(n)$ equals the probability that $q$ independent random permutations of $[n]$ have equal orders. Because of this, we will use ``collision probability'' to refer to $P_q(n)$ beyond the case $q = 2$; we will abuse the term ``collision entropy'' in a similar way.

Our first result is an anti-concentration estimate for $\ord(\pi_n)$ which confirms that the lower bound \eqref{eq:lower_bound_cycle} is asymptotically tight. We also obtain a structural description of exceptionally popular orders, i.e.\ values of $m$ for which $p_n(m)$ is close to the maximum. These turn out to be precisely the orders considered by Acan et al.\ in their disproof that $P_2(n) = O(1/n^2)$.

\begin{theorem}
\label{thm:asymp_for_max}
We have the asymptotic $P_{\infty}(n) \sim 1/n$. Moreover, if $n$ is sufficiently large, then any $m$ such that $p_n(m) \geq 1/n$ is of the form $n-k$ for some $k \in K_n$.
\end{theorem}

The second result identifies the mode of the distribution of $\ord(\pi_n)$, i.e.\ the exact value of $m$ for which $p_n(m)$ attains its maximum.

\begin{theorem}
\label{thm:equality_case}
For all sufficiently large $n$ and any $m$, we have $p_n(m) = P_{\infty}(n)$ if and only if $m = n - \max K_n$.
\end{theorem}

\begin{remark}
\label{rem:suff_large}
In Theorems \ref{thm:asymp_for_max} and \ref{thm:equality_case}, one cannot drop the assumption that $n$ is sufficiently large; counterexamples do exist for small values of $n$. One could, in principle, extract from our arguments an explicit bound on how large $n$ needs to be. However, what one would get is most probably not small enough to allow for a naive computational check of the remaining cases.
\end{remark}

As a consequence of our results on the most probable order, we can deduce a sharp asymptotic for the min-entropy of the order.

\begin{corollary}
\label{cor:min_ent_asymp}
We have $H_{\infty}(\ord(\pi_n)) = (1+O(\frac{1}{n}))\log n$.
\end{corollary}

We move on to discuss our results on the collision probability, which are the most substantial contribution of the paper. The results of Acan et al.\ \cite{acan-burnette-eberhard-schmutz-thomas} imply that $n^2P_2(n) \geq \log^*n - O(1)$ for infinitely many $n$, where $\log^*$ is the iterated logarithm function.\footnote{Strictly speaking, \cite[Theorem 2.1]{acan-burnette-eberhard-schmutz-thomas} states a bound which is weaker by a constant factor, but it is not difficult to modify the proof to obtain the above-stated bound.} Their argument carries over to the general case and yields the same lower bound for $n^qP_q(n)$, for arbitrary $q \in (1,\infty)$. In the following result, we show that this is actually best possible by giving an upper bound of the same shape for all values of $n$.

\begin{theorem}
\label{thm:collision_prob_upper_bound}
Let $q \in (1,\infty)$ be fixed. Then for all $n$ we have
\[
    P_q(n) \leq \frac{\log^*n + O(1)}{n^q}.
\]
\end{theorem}

Turning to questions about the lower limit behaviour of the collision probability, our next two results confirm the prediction that
\[
    \liminf_{n\to\infty}n^2P_2(n) < +\infty.
\]
The first one achieves this by identifying explicit subsequences along which the function converges. More specifically, it confirms the conjecture from \cite{eberhard-blog} that convergence holds along shifted factorials.

\begin{theorem}
\label{thm:collision_prob_explicit}
Let $q \in (1,\infty)$ and $D \in \Z$ be fixed. For $k \in \N$ define $n_k \vcentcolon= k! + D$. Then there exists a constant $\beta_{q,D} > 0$ such that
\[
    \lim_{k\to\infty}n_k^qP_q(n_k) = \beta_{q,D}.
\]
Moreover, we have the explicit expression
\begin{equation}\label{eq:collision_prob_const}
    \beta_{q,D} = \sum_{x \geq 0}\sum_{m\in\N}\P\Bigl(\frac{\ord(\pi_x)}{\gcd(\ord(\pi_x),x-D)} = m\Bigr)^q.
\end{equation}
\end{theorem}

The second result shows that the function $n^qP_q(n)$ is bounded on average. As a consequence, it follows that $n^qP_q(n) \ll 1$ holds fairly often, namely for all $n$ in a set of positive lower density.

\begin{theorem}
\label{thm:collision_prob_average}
For any fixed $q \in (1,\infty)$, we have
\[
    \frac{1}{N}\sum_{n=1}^{N}n^qP_q(n) = O(1).
\]
\end{theorem}

As a corollary of our results about the collision probability, we obtain an asymptotic for the Rényi $q$-entropy with an error term of best possible shape.

\begin{corollary}
\label{cor:collision_ent_asymp}
For any fixed $q \in (1,\infty)$, we have 
\[
    H_q(\ord(\pi_n)) = \frac{q\log n + O(\log(\log^*n))}{q-1}.
\]
\end{corollary}

Our final theorem establishes an asymptotic for the Shannon entropy of $\ord(\pi_n)$.

\begin{theorem}
\label{thm:shannon_entropy}
We have the asymptotic $H_1(\ord(\pi_n)) \sim \frac{1}{2}\log^2n$.
\end{theorem}

Our results show that, for $q > 1$, the behaviour of the Rényi $q$-entropy of $\ord(\pi_n)$ is tightly governed by arithmetic properties of $n$. In particular, the dominant contribution comes from values of the order that have a rather special arithmetic structure. These values lie close to $n$, far below the typical range of the order. In contrast, in the case of Shannon entropy, the bulk of the contribution comes exactly from the range predicted by the Erdős--Turán law.

Roughly speaking, our strategy for estimating $p_n(m)$ is analytic and draws inspiration from the anatomy of integers and permutations. We use distributional results about the order and the number of cycles to reduce matters to arithmetic considerations about $m$ and $n$. These issues are then addressed using tools from analytic number theory such as classical results on the distribution of primes and estimates for variants of the divisor function. Our approach to estimating the Rényi entropy crucially exploits the fact that the cycle lengths of a random permutation can be generated via a Markov process, a consequence of which is a recursive formula for $p_n(m)$. Furthermore, a key role is played by bounds for the number of divisors in short intervals. For a more detailed overview of our arguments, we refer the reader to Section \ref{sec:outline}.

\subsection{Organisation}\label{subsec:organisation}

The rest of the paper is organised as follows. In Section \ref{sec:prelim}, we record and develop various preliminary results necessary for the later arguments. Section \ref{sec:outline} contains an outline of the subsequent arguments. In Section \ref{sec:min_entropy}, we prove our results concerning the min-entropy, i.e.\ the most probable order, namely Theorems \ref{thm:asymp_for_max}, \ref{thm:equality_case} and Corollary \ref{cor:min_ent_asymp}. Section \ref{sec:collision_entropy} is devoted to the collision entropy and contains the proofs of Theorems \ref{thm:collision_prob_upper_bound}, \ref{thm:collision_prob_explicit} and \ref{thm:collision_prob_average}. In fact, we deduce these results from a more general ``master theorem'' which gives an expression for the collision probability in terms of an arithmetic quantity tied to $n$. Finally, in Section \ref{sec:shannon_entropy}, we address the Shannon entropy by proving Theorem \ref{thm:shannon_entropy}. 

\subsection{Notation}\label{subsec:notation}

We use standard asymptotic notation. Given quantities $A$ and $B$, we write $A \ll B$ to mean $A \leq O(B)$, that is, there is an absolute constant $C > 0$ such that $|A| \leq C|B|$; this is equivalent to $B \gg A$, i.e.\ $B \geq \Omega(A)$. If $A \ll B$ and $A \gg B$ hold simultaneously, we write $A \asymp B$. For functions $f, g \colon \N \to \R$, we write $f(n) = o(g(n))$ and $f(n) \sim g(n)$ to mean $\lim_{n\to\infty}f(n)/g(n) = 0$ and $\lim_{n\to\infty}f(n)/g(n) = 1$ respectively. We also write $f = \Theta(g)$ if $f = O(g)$ and $f = \Omega(g)$, which is equivalent to $f \asymp g$.

All logarithms are to base $e$ unless otherwise stated. For $k \in \N_0$, we write $\log^{(k)}$ for the $k$-fold iterated logarithm. Given $x > 0$, we also define the iterated logarithm of $x$, written $\log^*x$, to be the least $k$ such that $\log^{(k)}x \leq 1$. Whenever we write expressions involving a fixed number of iterated logarithms, such as $\log\log$, $\log\log\log$ etc., we implicitly assume that the relevant arguments are appropriately large.

As usual, we abbreviate the set $\{1,\ldots,n\}$ to $[n]$. We will primarily be concerned with intervals of integers, so it will be convenient to abuse the notation $(a,b]$ to denote the set of integers $x$ such that $a < x \leq b$, and similarly for $[a,b)$.

\subsection{Acknowledgements}\label{subsec:thanks}

This work was supported by the Croatian Science Foundation under the project number HRZZ-IP-2022-10-5116 (FANAP). The author would like to thank Rudi Mrazović for providing feedback on an earlier draft and Sean Eberhard for many useful comments and remarks.

\section{Preliminaries}\label{sec:prelim}

\subsection{Estimates for arithmetic functions}\label{subsec:nt_facts}

Given a positive integer $m$, we use the standard notation $\tau(m)$, $\sigma(m)$ and $\omega(m)$ to denote the number of positive divisors, sum of positive divisors and the number of distinct prime factors of $m$, respectively. An important role in the paper is played by the sum-of-reciprocal-divisors function
\[
    h(m) \vcentcolon= \sum_{d\mid m}\frac{1}{d}.
\]
This function represents the harmonic weight of the divisors of $m$ and is related to the sum of divisors function via $h(m) = \sigma(m)/m$. We will require estimates for these and several other arithmetic functions.

We start with the divisor bound, which follows from Theorem 2 in \cite[Chapter I.5]{tenenbaum}.

\begin{lemma}
\label{lm:divisor_bound}
We have $\tau(m) \leq \exp\Bigl(O\Bigl(\frac{\log m}{\log\log m}\Bigr)\Bigr)$. In particular, $\tau(m) \leq m^{o(1)}$.
\end{lemma}

Besides a bound on the number of divisors, we will also require a bound on their sum. The following result is contained in Theorem 5 of \cite[Chapter I.5]{tenenbaum}; here and in what follows, $\gamma \approx 0.577$ denotes the Euler--Mascheroni constant.

\begin{lemma}
\label{lm:sum_of_divisors}
We have $\sigma(m) \leq (1+o(1))e^{\gamma}m\log\log m$.
\end{lemma}

We will actually use Lemma \ref{lm:sum_of_divisors} in the form of the following corollary, which provides an estimate for $h(m)$.

\begin{corollary}
\label{cor:reciprocal_divisors}
We have $h(m) \leq (1+o(1))e^{\gamma}\log\log m$.
\end{corollary}

It is a standard fact that the number of distinct prime factors obeys the estimate
\[
    \omega(m) \leq (1+o(1))\frac{\log m}{\log\log m};
\]
see e.g.\ Theorem 3 (i) in \cite[Chapter I.5]{tenenbaum}. We will need such an estimate with a good error term. Specifically, we will use the following result, which appears as \cite[Theorem 12]{robin}, and is derived from (an explicit version of) the prime number theorem.

\begin{lemma}
\label{lm:prime_divisor_bound}
We have
\[
    \omega(m) \leq \frac{\log m}{\log\log m} + O\Bigl(\frac{\log m}{(\log\log m)^2}\Bigr).
\]
\end{lemma}

We will require Chebyshev's estimates (see e.g.\ Theorem 10 and Corollary 10.1 in \cite[Chapter I.2]{tenenbaum}). Recall that the prime counting function and the Chebyshev functions are defined as
\[
    \pi(x) \vcentcolon= \sum_{p\leq x}1, \quad \vartheta(x) \vcentcolon= \sum_{p\leq x}\log p, \quad \psi(x) \vcentcolon= \sum_{p \leq x}\lfloor\log_px\rfloor\log p.
\]
Here and throughout, the variable $p$ denotes a prime.

\begin{lemma}
\label{lm:chebyshev}
We have $\vartheta(x) \asymp \psi(x) \asymp x$ and $\pi(x) \asymp x/\log x$. In particular,
\[
    \prod_{p\leq x}p = \exp(\Theta(x)), \quad \lcm(1,\ldots,\lfloor x\rfloor) = \exp(\Theta(x)).
\]
\end{lemma}

We will also require Mertens' second theorem (see e.g.\ Theorem 9 in \cite[Chapter I.1]{tenenbaum}).

\begin{lemma}
\label{lm:mertens_second}
We have
\[
    \sum_{p \leq x}\frac{1}{p} = \log\log x + O(1).
\]
\end{lemma}

Putting together Lemmas \ref{lm:prime_divisor_bound}, \ref{lm:chebyshev} and \ref{lm:mertens_second}, we obtain the following.

\begin{corollary}
\label{cor:rec_prime_div}
For any $m \in \N$, we have
\[
    \sum_{p\mid m}\frac{1}{p} \leq \log\log\log m + O(1).  
\]
\end{corollary}
\begin{proof}
    Let $p_1 < \ldots < p_k$ be the distinct prime factors of $m$. For each $j \in \N$, let $q_j$ denote the $j$-th smallest prime. By Lemma \ref{lm:prime_divisor_bound}, we have $k \ll \log m/\log\log m$. By Lemma \ref{lm:chebyshev}, we have $k = \pi(q_k) \gg q_k/\log q_k$, whence $q_k \ll k\log k \ll \log m$. Since $p_j \geq q_j$ for all $j \in [k]$, it follows using Lemma \ref{lm:mertens_second} that
    \[
        \sum_{j=1}^{k}\frac{1}{p_j} \leq \sum_{j=1}^{k}\frac{1}{q_j} \leq \log\log q_k + O(1) \leq \log\log\log m + O(1),
    \]
    as required.
\end{proof}

The following simple lemma concerns the number of divisors in a short interval. It is based on the observation that numbers that are very close together cannot have many common divisors, and hence their least common multiple must be large. Several similar results appear in the literature, for example \cite[Theorem 3.1]{gabdullin} or \cite[Theorem 1]{granville-jimenez-urroz}. As usual, $\tau(m;a,b)$ stands for the number of divisors of $m$ in the interval $(a,b]$.

\begin{lemma}
\label{lm:div_in_int_simple}
If $a, m, t, r \in \N$ are such that $mt^{r^2} \leq a^r$, then $\tau(m;a,a+t) < r$.
\end{lemma}
\begin{proof}
    Suppose not, and let $d_1, \ldots, d_r$ be divisors of $m$ such that $a < d_1 < \ldots < d_r \leq a + t$. Then
    \[
        m \geq \mathrm{lcm}(d_1,\ldots,d_r) \geq \frac{\prod_{j=1}^{r}d_j}{\prod_{1\leq i < j\leq r}\gcd(d_i,d_j)} \geq \frac{\prod_{j=1}^{r}d_j}{\prod_{1\leq i < j\leq r}(d_j-d_i)} > \frac{a^r}{t^{\binom{r}{2}}},
    \]
    which is a contradiction.
\end{proof}

We end this subsection with another estimate for the number of divisors in short intervals, which might be of independent interest. This time the precise location of the interval plays no role; only its length matters.

\begin{lemma}
\label{lm:div_in_int_hard}
There exists an absolute constant $A > 0$ such that the following holds. For any $a,m,t \in \N$, we have $\tau(m;a,a+t) \leq (A\log m/\ell)^{\ell}$, where $\ell = \min(\lceil A\log t/\log\log t\rceil, \omega(m))$.
\end{lemma}
\begin{proof}
    Let $m = \prod_{i=1}^{k}p_i^{\alpha_i}$ be the prime factorisation of $m$. Split the divisors $d$ of $m$ from $(a,a+t]$ into two groups: those with $\omega(d) \leq \ell$ and those with $\omega(d) > \ell$. If $\ell = \omega(m)$, then the latter group is empty. Otherwise, by Lemma \ref{lm:chebyshev}, the product of the $\ell$ smallest primes is at least $t$. Thus, for any $\ell$ primes from $\{p_1,\ldots,p_k\}$, there can be at most one divisor of $m$ in $(a,a+t]$ which is divisible by all of these primes. Hence, by Lemma \ref{lm:prime_divisor_bound}, the latter group of divisors has size at most
    \[
        \binom{k}{\ell} \leq \Bigl(\frac{ek}{\ell}\Bigr)^{\ell} \leq O\Bigl(\frac{\log m}{\ell\log\log m}\Bigr)^{\ell},
    \]
    which is even smaller than what we need. The former group, on the other hand, has size at most
    \[
        \sum_{1\leq i_1<\ldots<i_{\ell}\leq k}\prod_{j=1}^{\ell}(\alpha_{i_j}+1) \leq \frac{1}{\ell!}\Bigl(\sum_{i=1}^{k}(\alpha_i+1)\Bigr)^{\ell} \leq O\Bigl(\frac{1}{\ell}\sum_{i=1}^{k}\alpha_i\Bigr)^{\ell}.
    \]
    The conclusion now follows since $\sum_{i=1}^{k}\alpha_i$, the number of prime factors of $m$ counted with multiplicity, is $O(\log m)$ (see e.g.\ Theorem 3 (ii) in \cite[Chapter I.5]{tenenbaum}).
\end{proof}

\subsection{Cycle type and order of random permutations}\label{subsec:cycle_type}

It will be important to recall that the cycle type of a random permutation can be sampled as follows. Given a positive integer $n$, consider the Markov chain $(Z_j^{(n)})_{j\geq0}$ with state space $\N_0$, initial distribution $\P(Z_0^{(n)} = n) = 1$ and transition probabilities
\[
    \P(Z_{j+1}^{(n)}=u\mid Z_j^{(n)}=v) = \begin{cases}1/v & \text{if } u < v\\1 & \text{if } u = v = 0\\0 & \text{otherwise}\end{cases}.
\]
In other words, the Markov chain starts at $n$ and at each step it jumps to a uniformly random smaller non-negative integer; once it arrives at $0$, it remains there for all time. Denoting the time of arriving at $0$ by
\[
    T^{(n)} \vcentcolon= \min\{j\geq 0 \mid Z_j^{(n)} = 0\},
\]
we have the following well-known fact (see e.g.\ \cite[\S1.1]{arratia-barbour-tavare}).

\begin{fact}
\label{fact:markov_chain}
The cycle type of $\pi_n$ has the same distribution as the multiset 
\[
    \{Z_j^{(n)}-Z_{j+1}^{(n)}\mid 0 \leq j < T^{(n)}\}.
\]
In particular, $\ord(\pi_n)$ has the same distribution as $\lcm\{Z_j^{(n)}-Z_{j+1}^{(n)}\mid 0 \leq j < T^{(n)}\}$.
\end{fact}

Conditioning on the first step of $(Z_j^{(n)})_{j\geq0}$, we obtain the following expression for $p_n(m)$ of a recursive nature.

\begin{corollary}
\label{cor:recursion}
For any $m, n \in \N$ we have
\begin{equation}\label{eq:recursion}
    \P(\ord(\pi_n) = m) = \frac{1}{n}\sum_{\substack{0\leq x<n\\n-x\mid m}}\P(\lcm(\ord(\pi_x), n-x) = m).
\end{equation}
In particular, for any $m, n \in \N$ we have
\[
    \P(\ord(\pi_n) \mid m) \leq \frac{\tau(m)}{n}.
\]
\end{corollary}

Along with Corollary \ref{cor:recursion}, our main tools are two local limit laws for the joint distribution of the order and the number of cycles. We will use these results to show that the probability of achieving a particular order $m$ is small. In doing so, we rely on the following rough dichotomy: either $m$ is not too large, so there is little chance that each cycle length is a divisor of $m$, or $m$ is large, which forces the cycle lengths to be divisible by many large prime powers.

We deduce our first local limit law as a consequence of the following more general result. Before stating it, we quickly set up some terminology. For a permutation $\pi \in S_n$, we define $c(\pi)$ to be the number of cycles in $\pi$. For an arbitrary set $I \subseteq \N$, we say $\pi$ is \emph{$I$-restricted} if the length of each cycle in $\pi$ belongs to $I$.

\begin{lemma}
\label{lm:cycle_lengths}
For any $\ell, n \in \N$ and $I \subseteq \N$ we have
\[
    \P(c(\pi_n) = \ell,\ \text{$\pi_n$ is $I$-restricted}) \leq \frac{\Bigl(\sum_{i\in I}1/i\Bigr)^{\ell-1}}{n(\ell-1)!}.
\]
\end{lemma}

Lemma \ref{lm:cycle_lengths} is a special case of \cite[Theorem 1.5]{ford}, a general local limit law for counts of cycle lengths based on the Poisson heuristic. At the same time, it can be obtained by a straightforward generalisation of the argument behind \cite[Lemma 4.1]{acan-burnette-eberhard-schmutz-thomas}, which corresponds to the case $I = [n]$. Taking $I$ to be the set of divisors of $m$, we obtain our first local limit law.
\begin{corollary}
\label{cor:cycle_divisors}
For any $\ell, m, n \in \N$ with $\ell > 1$ we have
\[
    \P(c(\pi_n) = \ell,\ \ord(\pi_n) \mid m) \leq \frac{1}{n}\Bigl(\frac{eh(m)}{\ell-1}\Bigr)^{\ell-1}.
\]
\end{corollary}

By taking $I = [n]$ in Lemma \ref{lm:cycle_lengths}, one can derive tail bounds for the number of cycles. We will use the following result, which is a special case of \cite[Theorem 1.7]{ford}. Since $\ord(\pi_n) \leq n^{c(\pi_n)}$, this result also provides an upper tail bound for $\ord(\pi_n)$; we will exploit this observation in several later instances.

\begin{corollary}
\label{cor:tail_bounds}
For $\lambda > 0$, define $Q(\lambda) \vcentcolon= \lambda\log\lambda - \lambda+1 \geq 0$. Then for any $0 < \lambda_1 \leq 1$ and $\lambda_2 > 1$ we have
\[
    \P(c(\pi_n) \leq \lambda_1\log n) \ll n^{-Q(\lambda_1)}, \quad \P(c(\pi_n) \geq \lambda_2\log n) \ll n^{-Q(\lambda_2/2)}.
\]
\end{corollary}

Our second local limit law is based on Fact \ref{fact:markov_chain} and can essentially be read out of the proof of \cite[Lemma 5.1]{acan-burnette-eberhard-schmutz-thomas}. For the sake of completeness, we provide a proof.

\begin{lemma}
\label{lm:cycles_and_primes}
For any $\ell, m, n \in \N$ we have
\[
    \P(c(\pi_n) = \ell,\ m \mid \ord(\pi_n)) \leq \frac{\ell^{\omega(m)}}{m}.
\]
\end{lemma}
\begin{proof}
    Let $m = \prod_{i=1}^{r}p_i^{\alpha_i}$ be the prime factorisation of $m$, where $r = \omega(m)$. Then $m$ divides the order if and only if for each $i \in [r]$ there exists a cycle of length divisible by $p_i^{\alpha_i}$. Hence, by Fact \ref{fact:markov_chain}, the probability on the left-hand side can be expressed as
    \[
        \P\Bigl(\{T^{(n)} = \ell\} \cap \bigcap_{i=1}^{r}\bigcup_{j=1}^{\ell}\{p_i^{\alpha_i} \mid Z_{j-1}^{(n)} - Z_j^{(n)}\}\Bigr) = \P\Bigl(\{T^{(n)} = \ell\} \cap \bigcup_{f \in [\ell]^{[r]}}\bigcap_{i=1}^{r}\{p_i^{\alpha_i} \mid Z_{{f(i)}-1}^{(n)} - Z_{f(i)}^{(n)}\}\Bigr).
    \]
    By the union bound, this is at most
    \[
        \sum_{f\in [\ell]^{[r]}}\P\Bigl(\{T^{(n)} = \ell\} \cap \bigcap_{i=1}^{r}\{p_i^{\alpha_i} \mid Z_{f(i)-1}^{(n)} - Z_{f(i)}^{(n)}\}\Bigr) \leq \sum_{f\in[\ell]^{[r]}}\P\Bigl(\bigcap_{j=1}^{\ell}\Bigl\{Z_{j-1}^{(n)} - Z_j^{(n)} \in q_{f,j}\N\Bigr\}\Bigr),
    \]
    where $q_{f,j} \vcentcolon= \prod_{i\in f^{-1}(j)}p_i^{\alpha_i}$. It thus suffices to show that each summand in the last sum is at most $1/m$. To see this, fix $f \colon [r] \to [\ell]$ and express the probability in question as
    \[
        \prod_{j=1}^{\ell}\P\Bigl(Z_{j-1}^{(n)} - Z_j^{(n)} \in q_{f,j}\N\ \Big|\ \bigcap_{j'<j}\{Z_{j'-1}^{(n)}-Z_{j'}^{(n)} \in q_{f,j'}\N\}\Bigr).
    \]
    For any $x \in [n]$, if we condition on $Z_{j-1}^{(n)} = x$, the difference $Z_{j-1}^{(n)} - Z_j^{(n)}$ becomes uniformly distributed on $[x]$. Hence, by the Markov property, the $j$-th factor in the above product is at most $1/q_{f,j}$. The desired conclusion now follows since $\prod_{j=1}^{\ell}q_{f,j} = m$.
\end{proof}

Even though we will formally only use Fact \ref{fact:markov_chain} through Corollary \ref{cor:recursion} and Lemma \ref{lm:cycles_and_primes}, it will be helpful to continue thinking about the cycle type of $\pi_n$ in terms of the random process described at the beginning of this subsection. In particular, we will occasionally refer to the Markov chain $(Z_j^{(n)})_{j\geq0}$ in informal descriptions of later arguments.

\section{Outline}\label{sec:outline}

We start by outlining the proof of Theorem \ref{thm:asymp_for_max}. For a given positive integer $m$, our aim is to upper bound the probability that $\ord(\pi_n)$ equals $m$. We may assume that $m \leq e^{O((\log n)^2)}$, as otherwise standard upper tail bounds on the order imply that $p_n(m) \leq o(1/n)$. In particular, by Lemma \ref{lm:sum_of_divisors}, $h(m) \ll \log\log n$. The key observation is that, as soon as $c(\pi_n)$ is significantly larger than $\log\log n$, Corollary \ref{cor:cycle_divisors} implies that the probability that $\ord(\pi_n)$ equals $m$ is $o(1/n)$. Therefore, it remains to deal with the case when the number of cycles is at most $O(\log\log n)$. The lower tail bound from Corollary \ref{cor:tail_bounds} already gives a bound of the form $n^{-1+o(1)}$ for the probability of this event, however, this is not enough for our purposes. Instead, we exploit the fact that the number of cycles is $(\log m)^{o(1)}$ via Lemma \ref{lm:cycles_and_primes} and Lemma \ref{lm:prime_divisor_bound}. This yields a bound of the form $m^{-1+o(1)}$, which is $o(1/n)$ as soon as $m \geq n^{1+\varepsilon}$ for some $\varepsilon > 0$. Hence, it remains to deal with $m \leq n^{1+o(1)}$, and this we do using Corollary \ref{cor:recursion}. Since $m$ is so small, it follows from Lemmas \ref{lm:divisor_bound} and \ref{lm:div_in_int_simple} that only a single value of $x$ can contribute significantly to the sum appearing in \eqref{eq:recursion}. In particular, we obtain the desired upper bound $p_n(m) \leq (1+o(1))/n$. Some additional work shows that $p_n(m)$ being close to the maximum forces $n-x$ to be divisible by $\lcm(1,\ldots,x)$, which completes the proof of Theorem \ref{thm:asymp_for_max}.

Given Theorem \ref{thm:asymp_for_max}, it is a fairly short step to Theorem \ref{thm:equality_case}. For $k \in K_n$, one expects the probability that $\ord(\pi_n) = n-k$ to be dominated by the event that $\pi_n$ contains a cycle of length $n-k$. The latter happens with probability $1/(n-k)$, an increasing function of $k$. It therefore suffices to prove a local limit law confirming the prediction that $\P(\ord(\pi_n) = n-k)$ equals $1/(n-k)$ up to a suitably small error. This can be accomplished by generalising some of the existing arguments in the literature pertaining to the case $k = 0$.

The proofs of the results concerning the collision probability follow broadly similar lines, but require several new ideas. For the sake of exposition, we focus here on the case $q = 2$. We begin by sketching an argument leading to a bound of the form $n^{-2+o(1)}$. Even though such a bound was already proved in \cite{acan-burnette-eberhard-schmutz-thomas}, our proof is simpler and we will be able to reuse some of the ideas in later arguments. By a similar argument as in the proof of Theorem \ref{thm:asymp_for_max}, one can show that
\begin{equation}\label{eq:squared_bound}
    \P(c(\pi_n) \geq L,\ \ord(\pi_n) = m) \leq o(1/n^2),
\end{equation}
where $L \vcentcolon= C\log n/\log\log n$ for a sufficiently large constant $C > 0$. By dividing into cases according to whether or not $\pi_n$ and $\pi_n'$ contain at most $L$ cycles, we may bound
\[
    P_2(n) \leq \P(\ord(\pi_n)=\ord(\pi_n'),\ c(\pi_n),c(\pi_n') \leq L) + 2\P(\ord(\pi_n)=\ord(\pi_n'),\ c(\pi_n) > L).
\]
By conditioning on the order of $\pi_n'$ and using \eqref{eq:squared_bound}, the second term can be seen to be $o(1/n^2)$. On the other hand, the lower tail bound from Corollary \ref{cor:tail_bounds} implies that the first term is at most
\[
    \P(c(\pi_n) \leq L)\P(c(\pi_n') \leq L) \leq n^{-2+o(1)},
\]
whence \eqref{eq:equal_order_prob} follows.

The main bottleneck in the previous argument is the use of the lower tail bound for the number of cycles. Indeed, essentially any argument making direct use of this bound cannot give a good quantification of the $n^{o(1)}$ term, let alone one of the form $\log^*n$. Instead, our starting point is the recursive expression \eqref{eq:recursion}. Since the sum is supported on at most $\tau(m)$ values of $x$, an application of the Cauchy--Schwarz inequality gives
\begin{equation}\label{eq:cauch_schwarz}
    p_n(m)^2 \leq \frac{\tau(m)}{n^2}\sum_{0\leq x < n}\P(\lcm(\ord(\pi_x),n-x) = m)^2.
\end{equation}
Furthermore, using Lemma \ref{lm:cycles_and_primes} and Lemma \ref{lm:prime_divisor_bound}, one can show that for sufficiently large $m$, say $m > M$, we have $p_n(m) \leq o(1/n^2)$. Hence, the total contribution to $P_2(n)$ of all $m > M$ is $o(1/n^2)$. We may thus restrict attention to $m \leq M$; summing \eqref{eq:cauch_schwarz} over that range and exchanging the order of summation, we obtain that
\begin{equation}\label{eq:exchanged_sum}
    \sum_{m\leq M}p_n(m)^2 \leq \frac{\max_{m\leq M}\tau(m)}{n^2}\sum_{0 \leq x < n}\sum_{m \leq M}\P(\lcm(\ord(\pi_x),n-x) = m)^2.
\end{equation}
The inner sum can be interpreted as the collision probability for (a truncated version of) the random variable $\lcm(\ord(\pi_x),n-x)$. As such, it can be estimated in essentially the same way as was done for $\ord(\pi_n)$ in the argument sketched above. Indeed, we only ever used the information that $\ord(\pi_n)$ divides $m$, however, it was crucial for the application of Corollary \ref{cor:cycle_divisors} that the number of cycles was significantly larger than $\log\log m \asymp \log\log n$. One can thus obtain an estimate of the form $x^{-2+\varepsilon}$ as long as $x \geq (\log n)^C$, where $C > 0$ is sufficiently large depending on $\varepsilon$, and $\varepsilon > 0$ is arbitrary. The sum of $x^{-2+\varepsilon}$ converges, so one obtains a bound of $(\log n)^{O(1)}$ for the outer sum in \eqref{eq:exchanged_sum}. Since $M$ can be chosen so that numbers up to $M$ have at most $n^{o(1)}$ divisors, one recovers from \eqref{eq:exchanged_sum} a bound of the form $n^{-2+o(1)}$.

There are two important issues with the above approach, and we have to address both in order to obtain the desired iterated logarithm bound. The first, and more obvious one, is that any approach that naively applies the Cauchy--Schwarz inequality as in \eqref{eq:cauch_schwarz} cannot ultimately yield a bound better than $n^{-2+o(1)}$. To circumvent this, we use Lemmas \ref{lm:divisor_bound}, \ref{lm:div_in_int_simple} and \ref{lm:div_in_int_hard} to split the range of summation $[0,n)$ in \eqref{eq:recursion} into several intervals $I$ so that $m$ has few divisors in the interval $n-I$. For all but one interval, the saving coming from the corresponding tail of the series $\sum_{x}x^{-2+\varepsilon}$ will balance out the cost incurred by the divisor bound. This naturally brings us to the second issue, which is how to deal with the remaining interval, which has the form $[0,(\log n)^{O(1)})$. It turns out that even though $h(m)$ may be much larger than $\log x$ for $x$ in this range, the previous argument goes through whenever $x$ does not belong to a certain set of exceptional values for which $n-x$ has many divisors. This set can be viewed as an approximate version of $K_n$, and one can show that its size behaves similarly to that of $K_n$. With a little bit of extra work, one can then obtain Theorems \ref{thm:collision_prob_upper_bound}, \ref{thm:collision_prob_explicit} and \ref{thm:collision_prob_average}.

Finally, the proof of Theorem \ref{thm:shannon_entropy} follows rather different lines to those of the previous results. Here, our main tool is a weak form of the Erdős--Turán law, which was established in their earlier paper \cite{erdos-turan-groups-i}, and states that $\ord(\pi_n) = \exp((\frac{1}{2}+o(1))\log^2n)$ with probability $1-o(1)$. Hence, the idea is to show that, from the perspective of Shannon entropy, $\ord(\pi_n)$ behaves roughly like a uniform random variable on the scale $\exp((\frac{1}{2}+o(1))\log^2n)$. To make this precise, we require some pointwise control on the probability mass function $p_n$. For this purpose, we use Lemma \ref{lm:cycles_and_primes}, however this time the bounds on $\omega(m)$ provided by Lemma \ref{lm:prime_divisor_bound} are not sufficient. Instead, we make use of another result of Erdős and Turán \cite{erdos-turan-groups-ii}, which states that $\omega(\ord(\pi_n)) \sim \log n\log\log n$ with probability $1-o(1)$.

\section{Min-entropy}\label{sec:min_entropy}

The following proposition is the main stepping stone towards Theorem \ref{thm:asymp_for_max}.

\begin{proposition}
\label{prop:large_orders}
For any $\varepsilon > 0$, we have
\[
    \max_{m\geq n^{1+\varepsilon}}p_n(m) = o(1/n).
\]
\end{proposition}
\begin{proof}
    Fix $\varepsilon > 0$, let $n$ be sufficiently large in terms of $\varepsilon$ and suppose that $m \geq n^{1+\varepsilon}$. The event that $\pi_n$ has order $m$ can be decomposed into the following three events according to the number of cycles in $\pi_n$:
    \[
        F_1 \vcentcolon= \{c(\pi_n) \leq \lceil C_1\log\log n\rceil,\ \ord(\pi_n) = m\},
    \]
    \[
        F_2 \vcentcolon= \{\lceil C_1\log\log n\rceil < c(\pi_n) \leq C_2\log n,\ \ord(\pi_n) = m\},
    \]
    \[
        F_3 \vcentcolon= \{c(\pi_n) > C_2\log n,\ \ord(\pi_n) = m\},
    \]
    where $C_1, C_2 > 0$ are sufficiently large absolute constants. We estimate the probabilities of these events in turn. First, Corollary \ref{cor:tail_bounds} gives
    \[
        \P(F_3) \leq \P(c(\pi_n) > C_2\log n) = o(1/n)
    \]
    provided $C_2$ is large enough. Next, since the order of a permutation is at most the product of the lengths of its cycles, we have $\ord(\pi_n) \leq n^{c(\pi_n)}$. Thus, if $m > n^{C_2\log n}$, then $\P(F_2) = 0$. Otherwise, by Lemma \ref{lm:sum_of_divisors}, we have $h(m)\leq C\log\log n$ for some absolute constant $C > 0$. Hence, using Corollary \ref{cor:cycle_divisors}, we obtain
    \[
        \P(F_2) = \sum_{\lceil C_1\log\log n\rceil<\ell\leq C_2\log n}\P(c(\pi_n) = \ell,\ \ord(\pi_n) = m) \leq C_2\log n\cdot\frac{1}{n}\Bigl(\frac{Ce}{C_1}\Bigr)^{C_1\log\log n} = o(1/n)
    \]
    provided $C_1$ is large enough. Finally, by Lemma \ref{lm:cycles_and_primes} and Lemma \ref{lm:prime_divisor_bound}, for any $\ell \leq \lceil C_1\log\log n\rceil$ we have
    \[
        \P(c(\pi_n) = \ell,\ \ord(\pi_n) = m) \leq \exp\Bigl(O\Bigl(\frac{\log m}{\log\log n}\cdot\log\log\log n\Bigr) - \log m\Bigr) \leq \frac{1}{n^{1+\varepsilon/2}}.
    \]
    By summing over all $\ell$ in this range, it follows that
    \[
        \P(F_1) = \sum_{\ell\leq \lceil C_1\log\log n\rceil}\P(c(\pi_n) = \ell,\ \ord(\pi_n) = m) \leq \frac{\lceil C_1\log\log n\rceil}{n^{1+\varepsilon/2}} = o(1/n),
    \]
    which concludes the proof.
\end{proof}

We are now ready to prove Theorem \ref{thm:asymp_for_max}. 

\begin{proof}[Proof of Theorem \ref{thm:asymp_for_max}]
    In view of \eqref{eq:lower_bound_cycle}, it suffices to show that, under the assumption that $n$ is sufficiently large and $m$ satisfies $p_n(m) \geq 1/n$, we have $p_n(m) \leq (1+o(1))/n$ and $n - m \in K_n$. In particular, by Proposition \ref{prop:large_orders}, we may assume that $m \leq n^{4/3}$ say. The starting point is an application of Corollary \ref{cor:recursion}. By the first statement, we have
    \begin{equation}\label{eq:recursion_new}
        p_n(m) = \frac{1}{n}\sum_{\substack{0\leq x<n\\n-x\mid m}}\P(\lcm(\ord(\pi_x), n-x) = m),
    \end{equation}
    and by the second statement applied with $x$ in place of $n$,
    \begin{equation}\label{eq:bound_on_summand}
         \P(\lcm(\ord(\pi_x), n-x) = m) \leq \frac{\tau(m)}{x}
    \end{equation}
    whenever $0 < x < n$. It follows that the total contribution of all $x \geq n^{1/2}$ to the right-hand side of \eqref{eq:recursion_new} is at most
    \[
        \frac{1}{n}\cdot\tau(m)\cdot\frac{\tau(m)}{n^{1/2}} = \frac{\tau(m)^2}{n^{3/2}},
    \]
    which, by Lemma \ref{lm:divisor_bound}, is at most $O(n^{-4/3})$ say. On the other hand, by Lemma \ref{lm:div_in_int_simple}, there is at most one $x < n^{1/2}$ such that $n-x \mid m$. Thus, $m$ must have a unique such divisor $d \in (n-n^{1/2}, n]$, which satisfies
    \[
        p_n(m) \leq \frac{1}{n}\P(\lcm(\ord(\pi_{n-d}), d) = m) + O(n^{-4/3}).
    \]
    In particular, we have $p_n(m) \leq (1+o(1))/n$. Moreover, since $m$ was assumed to satisfy $p_n(m) \geq 1/n$, it follows that
    \begin{equation}\label{eq:small_prob}
        \P(\lcm(\ord(\pi_{n-d}), d) \neq m) \ll n^{-1/3}.
    \end{equation}
    We contend that this forces $d$ to be divisible by all positive integers less than or equal to $n-d$. Indeed, suppose this does not hold. Then there exists a prime $p$ such that the largest power of $p$ not exceeding $n-d$, call it $q$, does not divide $d$. In particular, by maximality of $q$, we have 
    \begin{equation}\label{eq:maximality}
        q^2 \geq pq > n-d.
    \end{equation}
    Let $E$ be the event that the order of $\pi_{n-d}$ is divisible by $q$. Then on $E$, the $p$-adic valuation of $\lcm(\ord(\pi_{n-d}), d)$ equals that of $q$, and on $E^c$, it is strictly less than that of $q$. Consequently, at least one of $E$, $E^c$ is contained in the event on the left-hand side of \eqref{eq:small_prob}, so $\min(\P(E),\P(E^c)) \ll n^{-1/3}$. But by \cite[Lemma 1]{erdos-turan-groups-ii}, we have the exact expression
    \[
        \P(E^c) = \prod_{j=1}^{\lfloor (n-d)/q\rfloor}\Bigl(1-\frac{1}{jq}\Bigr).
    \]
    Hence, using the union bound and \eqref{eq:maximality}, we obtain the approximation
    \begin{equation}\label{eq:approx_prob}
        \frac{1}{q} \leq \P(E) \leq \sum_{j=1}^{\lfloor (n-d)/q\rfloor}\frac{1}{jq} \ll \frac{\log q}{q}.
    \end{equation}
    Note that by \eqref{eq:small_prob}, we certainly have
    \[
        \P(\lcm(\ord(\pi_{n-d}),d) = m) \geq \frac{1}{2},
    \]
    so \eqref{eq:bound_on_summand} implies $n-d \leq 2\tau(m)$. By Lemma \ref{lm:divisor_bound}, this means that $q \ll n^{1/4}$ say, whence the lower bound \eqref{eq:approx_prob} implies $\P(E) \gg n^{-1/4}$. Thus, we cannot have $\P(E) \ll n^{-1/3}$, so the only remaining option is $\P(E^c) \ll n^{-1/3}$. In view of the upper bound \eqref{eq:approx_prob}, this means that $q \ll 1$. Hence, by \eqref{eq:maximality}, we also have $n-d \ll 1$. But then necessarily $\P(E^c) = 0$, which is absurd since $\P(\ord(\pi_{n-d}) = 1) > 0$. Therefore, the claim follows, so in particular $d = m$. We conclude that $m = n-k$ for some $k \in K_n$, thereby completing the proof.
\end{proof}

Theorem \ref{thm:equality_case} follows by combining Theorem \ref{thm:asymp_for_max} and the following proposition, which gives accurate control on the point probabilities $\P(\ord(\pi_n) = n-k)$ for $k \in K_n$.

\begin{proposition}
\label{prop:exact_order}
For any $k \in K_n$ we have
\[
    \P(\ord(\pi_n) = n-k) = \frac{1}{n-k} + \eta(n,k) + O(n^{-3+o(1)}),
\]
where we define
\[
    \eta(n,k) \vcentcolon= \begin{cases}0 & \text{if } k \in \{0,1\} \text{ or } 2^{\lfloor\log_2k\rfloor+1} \mid n-k\\\frac{2^{1-\lfloor\log_2k\rfloor}}{(n-k)^2} & \text{otherwise}\end{cases}.
\]
\end{proposition}
\begin{proof}
    The proof is a relatively straightforward adaptation of the arguments of Warlimont \cite{warlimont}, which deal with the case $k = 0$.\footnote{In fact, Warlimont considers the probability that the order divides $n$ instead of being exactly equal to $n$, but this distinction is not significant.} Hence, we will be fairly brief on the details. As in \cite{warlimont}, we start by using Cauchy's formula \cite[Theorem 1.2]{ford} to express
    \[
        \P(\ord(\pi_n) = n-k) = \frac{1}{n-k} + \sum_{\substack{m,m_1,\ldots,m_r \in \N_0\\m+\sum_{j=1}^{r}m_jd_j = n\\\lcm\{d_j\mid j\in[r],\ m_j > 0\} = n-k}}\frac{1}{m!}\prod_{j=1}^{r}\frac{1}{m_j!d_j^{m_j}},
    \]
    where $1 < d_1 < \ldots < d_r < n-k$ are the divisors of $n-k$. For each $i \in \N_0$, we let $T_i$ be the total contribution of the terms satisfying $\sum_{j=s+1}^{r}m_j = i$, where $s$ is the number of $j \in [r]$ for which $d_j < n^{1-\delta}$, and $\delta > 0$ is some small parameter. One can then proceed in the same way as in \cite{warlimont} to bound
    \[
        \sum_{i=3}^{\infty}T_i \ll (\tau(n-k)n^{\delta-1})^3.
    \]
    Furthermore, in analogy to \cite{warlimont}, one can establish that 
    \[
        T_0 \leq F(n,k), \quad T_1 \leq \tau(n-k)n^{\delta-1}F(n,k),
    \] 
    where we define
    \[
        F(n,k) \vcentcolon= n\sum_{m\geq A(n,k)}\frac{1}{m!} + 2^{-B(n,k)}\tau(n-k)\exp(\tau(n-k)),
    \]
    \[
        A(n,k) \vcentcolon= \frac{n}{6\tau(n-k)},\quad B(n,k) \vcentcolon= \frac{n^{\delta}}{6\tau(n-k)}.
    \]
    In a similar vein, the total contribution to $T_2$ of all terms apart from those with
    \begin{equation}\label{eq:special_terms}
        d_r = \frac{n-k}{2},\quad m_r = 2, \quad m_{s+1} = \ldots = m_{r-1} = 0
    \end{equation}
    is at most $O((\tau(n-k)n^{\delta-1})^2F(n,k))$. It therefore remains to show that the contribution of the terms satisfying \eqref{eq:special_terms} is precisely $\eta(n,k)$. Indeed, if $n$ is large enough, then we have $\tau(n-k) \leq n^{\delta/3}$ and hence 
    \[
        A(n,k) \geq \frac{1}{6}n^{1-\delta/3}, \quad B(n,k) \geq \frac{1}{6}n^{2\delta/3}.
    \] 
    Consequently, we may bound
    \[
        \sum_{i=3}^{\infty}T_i \ll n^{-3+4\delta}
    \]
    and also
    \[
        F(n,k) \ll \exp\Bigl(-\frac{1}{6}n^{1-\delta/3}\Bigr) + \exp\Bigl(-\frac{\log2}{6}n^{2\delta/3}+n^{\delta/3}+\frac{\delta}{3}\log n\Bigr),
    \]
    which is certainly $O(n^{-3})$. Since $\delta > 0$ is arbitrary, we obtain an error term of the desired form.
    
    To finish the proof, we carefully examine the terms that satisfy \eqref{eq:special_terms}. For such terms, we have $m+\sum_{j=1}^{s}m_jd_j = k$. In order to have $\lcm\{d_j \mid j\in [r],\ m_j > 0\} = n-k$, there must exist $j \in [s]$ such that $m_j > 0$ and $\nu_2(d_j) = \nu_2(n-k)$, where $\nu_2$ denotes $2$-adic valuation. For this to be possible, we need $\nu_2(n-k)$ to be equal to the maximum of $\nu_2(t)$ over all $t \in [k]$. In particular, if $k \in \{0,1\}$ or $\nu_2(n-k) > \lfloor\log_2k\rfloor$, this is not possible. Otherwise, the terms of interest are precisely those which in addition to \eqref{eq:special_terms} satisfy $m_j=1$ for the unique $j \in [s]$ such that $d_j = 2^{\lfloor\log_2k\rfloor}$. Their total contribution is easily seen to be
    \[
        \frac{1}{2\bigl(\frac{n-k}{2}\bigr)^2\cdot2^{\lfloor\log_2k\rfloor}} = \frac{2^{1-\lfloor\log_2k\rfloor}}{(n-k)^2},
    \]
    as desired.
\end{proof}

We can now prove Theorem \ref{thm:equality_case}.

\begin{proof}[Proof of Theorem \ref{thm:equality_case}]
    Assume $n$ is sufficiently large and let $k_0 \vcentcolon= \max K_n$. By Theorem \ref{thm:asymp_for_max}, it suffices to prove that $p_n(n-k_0) > p_n(n-k)$ for all $k \in K_n \setminus \{k_0\}$.
    Hence, by Proposition \ref{prop:exact_order}, it is enough to show that
    \begin{equation}\label{eq:final_ineq}
        \frac{k_0-k}{(n-k_0)(n-k)} + \eta(n,k_0)-\eta(n,k) \geq \frac{1}{(n-k)^2}.
    \end{equation}
    If $k \in \{0,1\}$, then $\eta(n,k) = 0$, so \eqref{eq:final_ineq} certainly holds. Hence, we may assume that $k \geq 2$, so in particular $\eta(n,k) \leq 1/(n-k)^2$. Since $\lcm(1,\ldots,k)$ divides both $n-k$ and $n-k_0$, it must divide $k_0-k$. Therefore, $k_0 - k \geq 2$, so the left-hand side  of \eqref{eq:final_ineq} is at least
    \[
        \frac{2}{(n-k_0)(n-k)} - \frac{1}{(n-k)^2} > \frac{1}{(n-k)^2},
    \]
    as desired.
\end{proof}

Finally, we are in a position to prove Corollary \ref{cor:min_ent_asymp}. Indeed, it follows from Lemma \ref{lm:chebyshev} that $\max K_n \ll \log n$. As a consequence of Theorem \ref{thm:equality_case} and Proposition \ref{prop:exact_order}, we thus obtain the more refined asymptotic 
\[
    P_{\infty}(n) = \frac{1}{n} + O\Bigl(\frac{\log n}{n^2}\Bigr).
\]
A short calculation now yields
\[
    H_{\infty}(\ord(\pi_n)) = \Bigl(1+O\Bigl(\frac{1}{n}\Bigr)\Bigr)\log n,
\]
which is what we wanted to prove. That the error term here is best possible can be seen by considering $n$ of the form $\lcm(1,\ldots,k)+k$ for $k \in \N$.

\section{Collision entropy}\label{sec:collision_entropy}

Throughout this section, we treat the parameter $q \in (1,\infty)$ as fixed; in particular, we will suppress the dependence of implied constants on $q$. We bound the collision probability for $\ord(\pi_n)$ in terms of the count of numbers in $[n]$ possessing an exceptionally rich arithmetic structure. In order to make this precise, recall that $h(m)$ denotes the sum of reciprocal divisors of $m$. Call a number $x \in \{0,1,\ldots,n-1\}$ \emph{exceptional} if
\begin{equation}\label{eq:except_defn}
    h\bigl(\gcd(\lcm(1,\ldots,x), n-x)\bigr) \geq \frac{c\log x}{\log\log x},
\end{equation}
where $c > 0$ is a suitably small fixed constant. Let $E_n$ be the set of all exceptional numbers. Note that if $x \in K_n$, then the left-hand side of \eqref{eq:except_defn} is at least $\Omega(\log x)$, so $K_n$ is a subset of $E_n$.

The following is the main result of this section, from which Theorems \ref{thm:collision_prob_upper_bound}, \ref{thm:collision_prob_explicit} and \ref{thm:collision_prob_average} will later be deduced.

\begin{theorem}
\label{thm:except_collision}
Let $q \in (1, \infty)$ be fixed. Then for all $n$ and $0 < X \ll \log n$ we have
\[
    n^qP_q(n) = \sum_{0\leq x < X}\sum_{m\in\N}\P(\lcm(\ord(\pi_x),n-x) = m)^q + R + \rho,
\]
where $0 \leq R \leq |E_n \cap [X,n)| + O(X^{-\Omega(1)})$ and $|\rho| \leq o(1)$. In particular, $n^qP_q(n) \leq |E_n|+O(1)$.
\end{theorem}
We now begin working towards a proof of Theorem \ref{thm:except_collision}; this will occupy the bulk of this section. To this end, it will be convenient to rescale the function $p_n$ by defining $\widetilde{p}_n \vcentcolon= np_n$. We also introduce some global parameters that will play an important role in our arguments. First, we define 
\[
    M = M(n) \vcentcolon= n^{\frac{C\log\log n}{\log\log\log n}},
\]
where $C > 0$ is a suitably large absolute constant. Next, we define the hierarchy of decreasing scales $X_j = X_j(n)$ for $0 \leq j \leq 4$ by
\[
    X_0 \vcentcolon= n, \quad X_1 \vcentcolon= \exp\Bigl(\frac{C_1\log n}{\log\log\log n}\Bigr), \quad X_2 \vcentcolon= \exp\Bigl(\frac{C_2\log n}{\log\log n}\Bigr), \quad X_3 \vcentcolon= (\log n)^{C_3}, \quad X_4 \vcentcolon= 0,
\]
where $C_1, C_2 > 0$ are suitable (large) constants, and $C_3$ is any fixed constant belonging to the range $(e^\gamma, 2)$. Note that we have $e^\gamma \approx 1.781 < 2$, so a suitable choice of $C_3$ indeed exists.\footnote{It seems plausible that our arguments could be refined so as not to rely on this numerical fact, however, we will not pursue this here.} As will soon become clear, the precise quantitative form of the parameters just introduced will be rather important for our arguments.

The first step is to eliminate very large orders by showing that they have a negligible contribution to the collision probability. Since the length of each jump of the Markov chain $(Z_j^{(n)})_{j\geq0}$ is a divisor of the order, this will limit the number of possibilities for each step.

\begin{lemma}
\label{lm:large_orders}
For all $m > M$ we have $p_n(m) \ll n^{-(q+1)}$. In particular,
\[
    \lVert \wt{p}_n\rVert_q^q = \lVert \wt{p}_n1_{(0, M]}\rVert_q^q + O(1/n).
\]
\end{lemma}
\begin{proof}
    The second statement follows from the first via the decomposition
    \[
        \lVert \wt{p}_n\rVert_q^q = \sum_{m\leq M}\wt{p}_n(m)^q + \sum_{m>M}\wt{p}_n(m)^q
    \]
    and the observation that the second sum may be upper bounded by
    \[
        \Bigl(\max_{m>M}\wt{p}_n(m)\Bigr)^{q-1}\sum_{m>M}\wt{p}_n(m) \leq n^q\max_{m>M}p_n(m).
    \]
    Hence, it suffices to prove the first statement. For a suitably large constant $B > 0$, we have
    \[
        p_n(m) \leq \P(c(\pi_n) > B\log n) + \sum_{\ell \leq B\log n}\P(c(\pi_n) = \ell,\ m \mid \ord(\pi_n)).
    \]
    If $B$ is large enough, Corollary \ref{cor:tail_bounds} implies that the first summand is $O(n^{-(q+1)})$. The other summands can be bounded via Lemma \ref{lm:cycles_and_primes}. Write $m = n^{\kappa}$ and $\omega(m) = (1+\delta)\frac{\log m}{\log\log m}$, where 
    \begin{equation}\label{eq:kappa_lower_bound}
        \kappa > \frac{C\log\log n}{\log\log\log n}
    \end{equation}
    and, by Lemma \ref{lm:prime_divisor_bound}, $\delta \leq O(1/\log\log m)$. For $\ell \leq B\log n$, we have $\log\ell \leq \log\log n + O(1)$, so it follows that
    \begin{align*}
        \P(c(\pi_n) = \ell,\ m \mid \ord(\pi_n)) &\leq \exp\Bigl((1+\delta)\frac{\log m}{\log\log m}\cdot\log\ell - \log m\Bigr)\\
        &= \exp\Bigl(-\frac{\log m}{\log\log m}(\log\log m - \log\ell - \delta\log\ell)\Bigr)\\
        &\leq \exp\Bigl(-\frac{\kappa\log n}{\log\log n + \log\kappa}\cdot(\log\kappa - O(1))\Bigr).
    \end{align*}
    The latter is a decreasing function of $\kappa$, and a short calculation using \eqref{eq:kappa_lower_bound} shows that choosing $C>q+2$ gives an upper bound of $O(n^{-(q+2)})$. Summing over all $\ell$ in this range then yields the desired bound $p_n(m) \leq O(n^{-(q+1)})$.
\end{proof}

\begin{remark}
\label{rem:large_orders}
The above proof is exactly the place where the full quantitative strength of Lemma \ref{lm:prime_divisor_bound} is required. Indeed, bounds for the error term available without recourse to the prime number theorem, such as
\[
    O\Bigl(\frac{\log m\log\log\log m}{(\log\log m)^2}\Bigr),
\]
are not sufficient to close the argument.
\end{remark}

One might think that the shape of the parameter $M$ is an artefact of the proof of Lemma \ref{lm:large_orders}, and that one already has $p_n(m) \leq o(n^{-q})$ for much smaller values of $m$. However, as the following result shows, this is not the case; in fact, Lemma \ref{lm:large_orders} is essentially best possible.

\begin{proposition}
\label{prop:large_prob}
For any $q \in (1,\infty)$ and $\varepsilon > 0$, there are infinitely many $n$ for which there exists
\[
    m \geq n^{(\frac{q-1}{2}-\varepsilon)\frac{\log\log n}{\log\log\log n}}
\]
such that $p_n(m)\gg n^{-q}$.
\end{proposition}

Proposition \ref{prop:large_prob} is not logically necessary for our main results, and its proof can be read independently of the rest of this section.

\begin{proof}[Proof of Proposition \ref{prop:large_prob}.]
    Let $k > \ell$ be positive integers and set $t \vcentcolon= k-\ell$; we will assume that $k,\ell, t$ are sufficiently large. For $j \in \N$ let $q_j$ denote the $j$-th smallest prime, and set $m \vcentcolon= \prod_{j=1}^{k}q_j$. Let $\Pi$ be the set of all partitions of $[k]\setminus[\ell]$ and let $\mathcal{P} \in \Pi$ be such a partition. Consider the quantity
    \[
        \Sigma(\mathcal{P}) \vcentcolon= \sum_{P\in\mathcal{P} \cup \{[\ell]\}}\prod_{j\in P}q_j.
    \]
    According to Cauchy's formula \cite[Theorem 1.2]{ford}, the probability that $\pi_{\Sigma(\mathcal{P})}$ has cycle type
    \[
        \Bigl\{\prod_{j\in P}q_j\ \Big|\ P \in \mathcal{P} \cup \{[\ell]\}\Bigr\}
    \]
    is exactly $1/m$. Moreover, if this happens, then $\pi_{\Sigma(\mathcal{P})}$ has order exactly $m$. Therefore, for $n \in \N$ we have the lower bound
    \[
        p_n(m) \geq \frac{|\{\mathcal{P} \in \Pi \mid \Sigma(\mathcal{P}) = n\}|}{m}.
    \]
    The idea now is to use the pigeonhole principle to locate many partitions $\mathcal{P}$ achieving the same value of $\Sigma(\mathcal{P})$. To carry out this strategy, we first have to restrict to partitions $\mathcal{P}$ such that $\Sigma(\mathcal{P})$ belongs to a narrower range. Luckily, the largest part of a typical partition $\mathcal{P} \in \Pi$ is of size asymptotically $e\log t$, so most partitions will satisfy this.\footnote{A much cruder bound on the size of the largest part would be sufficient.} More precisely, defining
    \[
        \mu(\mathcal{P}) \vcentcolon= \max_{P\in\mathcal{P}}|P|,
    \]
    the results of \cite{sachkov} imply that $\E_{\mathcal{P} \in \Pi} \mu(\mathcal{P}) \sim e\log t$. In particular, Markov's inequality implies that 
    \[
        \P_{\mathcal{P}\in\Pi}(\mu(\mathcal{P}) \leq 10\log t) \geq \frac{2}{3}.
    \]
    Note that if $\mu(\mathcal{P}) \leq 10\log t$, then for all $P \in \mathcal{P}$ we have $\prod_{j\in P}q_j \leq q_k^{10\log t}$, so $\Sigma(\mathcal{P})$ certainly belongs to the interval
    \[
        I \vcentcolon= \Bigl(\prod_{j=1}^{\ell}q_j,\prod_{j=1}^{\ell}q_j + tq_k^{10\log t}\Bigr].
    \]
    Hence, we have
    \[
        \sum_{n\in I}p_n(m) \geq \frac{|\{\mathcal{P} \in \Pi \mid \Sigma(\mathcal{P}) \in I\}|}{m} \geq \frac{|\Pi|}{m}\P_{\mathcal{P}\in\Pi}(\mu(\mathcal{P}) \leq 10\log t) \geq \frac{2B_t}{3m},
    \]
    where $B_t$ denotes the $t$-th Bell number. Thus, by averaging, there exists $n \in I$ such that
    \[
        p_n(m) \geq \frac{2B_t}{3m|I|} \gg \frac{B_t}{mtq_k^{10\log t}}.
    \]
    We would like the latter quantity to be at least $\Omega(n^{-q})$, or equivalently that the following holds:
    \[
        \log m + (10\log q_k + 1)\log t \leq q\log n + \log B_t + O(1).
    \]
    Using that $\log m = \vartheta(q_k)$, $\log n \geq \vartheta(p_{\ell})$ together with the facts that $\vartheta(q_j)$ and $q_j$ have the asymptotic form $j(\log j+\log\log j + O(1))$ (see \cite{robin}) as well as that ${\log B_t = t(\log t - \log\log t + O(1))}$ (see \cite{de-bruijn}), a tedious albeit straightforward calculation shows that one may take
    \[
        \ell = \Biggl\lceil\frac{(\frac{2}{q-1}+\delta)k\log\log k}{\log k}\Biggr\rceil,
    \]
    where $\delta > 0$ is any constant. Finally, taking $\delta$ small enough and using ${n = \exp((1+o(1))\ell\log\ell)}$, $m = \exp((1+o(1))k\log k)$, it follows that $m > n^{(\frac{q-1}{2}-\varepsilon)\frac{\log\log n}{\log\log\log n}}$, as desired.
\end{proof}

In view of the recursive expression \eqref{eq:recursion}, we decompose $\wt{p}_n = \sum_{j=1}^{4}\wt{p}_n^{(j)}$, where for each $j \in [4]$ and $m \in \N$ we define
\begin{equation}\label{eq:recursion_range}
    \wt{p}_n^{(j)}(m) \vcentcolon= \sum_{X_j \leq x < X_{j-1}}\P(\lcm(\ord(\pi_x), n-x) = m).
\end{equation}
The next lemma encapsulates the key Hölder manipulation.

\begin{lemma}
\label{lm:cauchy_schwarz}
Fix any $j \in [4]$, $n \in \N$, $N > 0$ and write
\[
    T \vcentcolon= \max_{m\leq N}\tau(m;n-X_{j-1},n-X_j).
\]
Then for any $q \in (1, \infty)$ we have
\[
    \lVert\wt{p}_n^{(j)}1_{(0, N]}\rVert_q^q \leq T^{q-1}\sum_{X_j\leq x < X_{j-1}}\sum_{m\leq N}\P(\lcm(\ord(\pi_x), n-x) = m)^q.
\]
Moreover, if $T \leq 1$, we may remove the factor $T^{q-1}$ from the right-hand side to obtain an equality.
\end{lemma}
\begin{proof}
    Taking $q$-th powers in \eqref{eq:recursion_range} and applying Hölder's inequality, noting that the summand is non-zero only if $n-x \mid m$, it follows that
    \[
        \wt{p}_n^{(j)}(m)^q \leq \tau(m;n-X_{j-1},n-X_j)^{q-1}\sum_{X_j \leq x < X_{j-1}}\P(\lcm(\ord(\pi_x), n-x) = m)^q.
    \]
    If the first factor on the right-hand side is at most $1$, then we may remove it to obtain an equality. The conclusion follows on summing over all $m \leq N$ and interchanging the order of summation.
\end{proof}

Lemma \ref{lm:cauchy_schwarz} reduces matters to two separate issues: bounding the number of divisors of potential orders in short intervals and understanding the collision probability for the random variable $\lcm(\ord(\pi_x),n-x)$. We address these issues in turn, starting from the second one. Since ${\lcm(\ord(\pi_x),n-x)}$ is a function of $\ord(\pi_x)$, its collision probability cannot be smaller than that of $\ord(\pi_x)$. However, unless $n-x$ has exceptionally rich multiplicative structure, one expects that taking the least common multiple with $n-x$ does not significantly affect the collision probability. The following lemma makes this idea precise by providing bounds depending on whether $x$ is exceptional or not.

\begin{lemma}
\label{lm:collision_prob}
Let $q \in (1,\infty)$ be fixed. There is an absolute constant $c' > 0$ such that for $x \in \{0,1,\ldots,n-1\}$ we have
\begin{equation}\label{eq:cases_bound}
    \sum_{m \leq M}\P(\lcm(\ord(\pi_x), n-x) = m)^q \ll \begin{cases}1/x^{1+c'} & \text{if $x$ is not exceptional}\\1/x^{c'} & \text{if $x \geq X_3$}\end{cases}.
\end{equation}
\end{lemma}
\begin{proof}
    We may assume throughout the proof that $x$, and hence also $n$, is sufficiently large. Write $W \vcentcolon= \lcm(\ord(\pi_x),n-x)$, suppressing the dependence on $n$ and $x$ for brevity. Let $B$ be a value of $\lambda_2$ in Corollary \ref{cor:tail_bounds} such that $Q(\lambda_2/2) \geq 2$. Suppose first that $x$ is not exceptional and set $L \vcentcolon= \lceil c'\log x\rceil$ for a suitably small constant $c' > 0$. Dividing into cases according to whether $c(\pi_x)$ belongs to $I_1 \vcentcolon= (0,L]$, $I_2 \vcentcolon= (L, B\log x]$ or $I_3 \vcentcolon= (B\log x,\infty)$ and applying Hölder's inequality, we may upper bound the left hand side of \eqref{eq:cases_bound} by
    \begin{equation*}
        3^{q-1}\sum_{j=1}^{3}T_j, \quad \text{where} \quad T_j \vcentcolon= \sum_{m\leq M}\P(W = m, c(\pi_x) \in I_j)^q.
    \end{equation*}
    For $j \in \{1,3\}$, we have
    \[
        T_j \leq \Bigl(\sum_{m \leq M}\P(W=m, c(\pi_x)\in I_j)\Bigr)^q \leq  \P(c(\pi_x) \in I_j)^q.
    \]
    Hence, by choice of $B$, we have $T_3 \ll x^{-2}$. Likewise, Corollary \ref{cor:tail_bounds} implies that by choosing $c' > 0$ small enough, one can achieve
    \[
        T_1 \ll 1/x^{1+c'}.
    \]
    It therefore remains to bound $T_2$. To this end, note that $c(\pi_x) \leq B\log x$ implies $\ord(\pi_x) \leq x^{B\log x}$, so we have
    \begin{equation*}
        T_2 \leq \max_{m \leq x^{B\log x}}\P\Bigl(\frac{\ord(\pi_x)}{\gcd(\ord(\pi_x),n-x)} = m,\ c(\pi_x) \in I_2\Bigr)^{q-1}.
    \end{equation*}
    If the event on the right-hand side happens, $\ord(\pi_x)$ must divide $m\gcd(\lcm(1,\ldots,x),n-x)$. Thus, by Corollary \ref{cor:cycle_divisors}, we have
    \[
        \P\Bigl(\frac{\ord(\pi_x)}{\gcd(\ord(\pi_x),n-x)} = m,\ c(\pi_x) \in I_2\Bigr) \leq \sum_{L<\ell\leq B\log x}\frac{1}{x}\Biggl(\frac{eh\bigl(m\gcd(\lcm(1,\ldots,x),n-x)\bigr)}{\ell-1}\Biggr)^{\ell-1}.
    \]
    The function $h$ is (completely) submultiplicative, i.e.\ satisfies $h(m_1m_2) \leq h(m_1)h(m_2)$ for all $m_1,m_2$. Note that for $m \leq x^{B\log x}$, Corollary \ref{cor:reciprocal_divisors} implies that $h(m) \ll \log\log x$. Hence, if $c > 0$ is chosen to be sufficiently small, we have
    \[
        h\bigl(m\gcd(\lcm(1,\ldots,x),n-x)\bigr) \leq h(m)\cdot\frac{c\log x}{\log\log x} \leq c''\log x,
    \]
    where $c'' > 0$ is an arbitrarily small absolute constant. Thus, by choosing $c''$ to be small enough, we may achieve that $T_2 \ll x^{-2}$ say. This concludes the argument in the case when $x$ is exceptional.

    We now move on to the case when $x$ is only assumed to be at least $X_3$. The argument in this case is similar; we will show that in fact
    \[
        \P(W = m) \ll 1/x^{\Omega(1)}
    \]
    uniformly for all $m \leq M$. To this end, note that the left-hand side in the above display is at most
    \begin{equation}\label{eq:except_decomp}
        \P(c(\pi_x) > B\log x) + \sum_{\ell \leq B\log x}\P(c(\pi_x) = \ell,\ \ord(\pi_x) \mid m).
    \end{equation}
    The first summand is at most $O(x^{-2})$, and the other terms apart from the one corresponding to $\ell = 1$ can be upper bounded using Corollary \ref{cor:cycle_divisors}:
    \begin{equation}\label{eq:except_divisors}
        \P(c(\pi_x) = \ell,\ \ord(\pi_x) \mid m) \leq \frac{1}{x}\Bigl(\frac{eh(m)}{\ell-1}\Bigr)^{\ell-1}.
    \end{equation}
    Note that we have 
    \[
        \log\log m \leq \log\log M \leq (1+o(1))\log\log n,
    \]
    so if $m$ is large enough, it follows from Corollary \ref{cor:reciprocal_divisors} that
    \[
        h(m) \leq \eta C_3\log\log n \leq \eta\log x,
    \]
    where $\eta \in (e^{\gamma}/C_3, 1)$ is any fixed constant. If $\ell = 1$, the left-hand side of \eqref{eq:except_divisors} is at most $1/x$, whereas if $\ell > 1$, then writing $\ell-1 = \kappa\log x$ for some $0 < \kappa < B$, we obtain that it is at most
    \[
        f(\kappa) \vcentcolon= \frac{1}{x}\Bigl(\frac{e\eta}{\kappa}\Bigr)^{\kappa\log x}.
    \]
    A simple application of calculus shows that this function has a maximum at $\kappa = \eta$, so we get a bound of the form $f(\eta) = x^{\eta-1}$. In conclusion, we have that \eqref{eq:except_decomp} is at most $O(x^{(\eta-1)/2})$, as desired.
\end{proof}

\begin{remark}
\label{rem:collision_prob}
For most applications of Corollary \ref{cor:reciprocal_divisors}, the statement that $h(m) \ll \log\log m$ is sufficient. The second part of the above proof is exactly the point where the value of the implied constant is important.
\end{remark}

Our next lemma bounds the number of divisors in intervals of the form $(n-X_{j-1}, n-X_j]$. With Lemmas \ref{lm:divisor_bound}, \ref{lm:div_in_int_simple} and \ref{lm:div_in_int_hard} in place, the proof is mostly a technical calculation.

\begin{lemma}
\label{lm:interval_bounds}
For all $j \in [3]$ we have
\[
    \max_{m\leq M}\tau(m;n-X_{j-1},n-X_j) = O(X_j^{c'/2}),
\]
where $c' > 0$ is the constant from the statement of Lemma \ref{lm:collision_prob}.
\end{lemma}
\begin{proof}
    Suppose that $m \leq M$, and assume without loss of generality that $n$ is sufficiently large. We consider the three cases in turn. If $j = 1$, then for some absolute constant $C' > 0$, by Lemma \ref{lm:divisor_bound} we have
    \[
        \tau(m;n-X_0,n-X_1) \leq \tau(m) \leq \exp\Bigl(\frac{C'\log M}{\log\log M}\Bigr) \leq \exp\Bigl(\frac{CC'\log n}{\log\log\log n}\Bigr) \leq X_1^{c'/2}
    \]
    provided $C_1$ is large enough, namely $C_1 \geq 2CC'/c'$. If $j = 2$, then setting $a = n-X_1$, $t = X_1-X_2$ in Lemma \ref{lm:div_in_int_hard}, we obtain that
    \begin{equation}\label{eq:divisor_bound}
        \tau(m;n-X_1,n-X_2) \leq \Bigl(\frac{A\log M}{\ell}\Bigr)^{\ell},
    \end{equation}
    where $\ell \vcentcolon= \min(\lceil A\log t/\log\log t \rceil, \omega(m))$. Note that for $n$ large enough, we have
    \[
        \lceil A\log t/\log\log t \rceil \leq \frac{2A\log X_1}{\log\log X_1} \leq \frac{4AC_1\log n}{\log\log n\log\log\log n} =\vcentcolon T.
    \]
    Since the function $x \mapsto (A\log M/x)^x$ is increasing on $(0, T]$, it follows that the right-hand side in \eqref{eq:divisor_bound} is at most
    \[
        \Biggl(\frac{\frac{AC\log n\log\log n}{\log\log\log n}}{T}\Biggr)^T \leq \exp\Bigl(\frac{10AC_1\log n}{\log\log n}\Bigr) \leq X_2^{c'/2}
    \]
    provided $C_2$ is chosen so that $C_2 \geq 20AC_1/c'$. Finally, we deal with the remaining case $j = 3$ by appealing to Lemma \ref{lm:div_in_int_simple}. Taking $a = n-X_2$, $t = X_2-X_3$ and $r = \lfloor (\log\log n)/(2C_2)\rfloor$, we have
    \[
        a^r/t^{r^2} > \Bigl(\frac{n-X_2}{X_2^r}\Bigr)^r > (n^{1/4})^{\frac{\log\log n}{4C_2}} > m,
    \]
    whence Lemma \ref{lm:div_in_int_simple} implies that
    \[
        \tau(m;n-X_2,n-X_3) < r \leq \frac{\log\log n}{2C_2} \leq X_3^{c'/2}.
    \]
    This concludes the proof.
\end{proof}

\begin{remark}
\label{rem:interval_bounds}
Numbers up to $M$ can have as many as $X_2' = \exp\bigl(\Omega(\frac{\log n\log\log\log n}{\log\log n})\bigr)$ divisors in $(0, n]$, which is just too many for Lemma \ref{lm:div_in_int_simple} to be applicable to the interval ${(n-X_2',n-X_3]}$. It therefore seems difficult to avoid considering the interval $[X_2,X_1)$ in our arguments. Furthermore, Lemma \ref{lm:div_in_int_simple} produces bounds on $\tau(m;n-X_1,n-X_2)$ only for $m \leq n^{\Omega(\log\log\log n)}$, which is why we need Lemma \ref{lm:div_in_int_hard} to deal with this interval.
\end{remark}

The following lemma casts the defining property of exceptional numbers \eqref{eq:except_defn} into a form which is slightly more convenient to work with.

\begin{lemma}
\label{lm:except_equiv}
If $x$ is exceptional, then
\begin{equation}\label{eq:rec_prime_div}
    \sum_{\substack{p \mid n-x\\p \leq x}}\frac{1}{p} \geq \log\log x - \log\log\log x - O(1)
\end{equation}
and hence
\begin{equation}\label{eq:rec_prime_not_div}
    \sum_{\substack{p \nmid n-x\\p \leq x}}\frac{1}{p} \leq \log\log\log x + O(1).
\end{equation}
\end{lemma}
\begin{proof}
    Set $m \vcentcolon= \gcd(\lcm(1,\ldots,x),n-x)$ and let $m = \prod_{i=1}^{k}p_i^{\alpha_i}$ be the prime factorisation of $m$. Using the multiplicativity of $h$ and the inequality $\log(1+t) \leq t$, we obtain that
    \[
        \log h(m) = \log\Biggl(\prod_{i=1}^{k}h(p_i^{\alpha_i})\Biggr) = \sum_{i=1}^{k}\log\Bigl(1+\sum_{j=1}^{\alpha_i}\frac{1}{p_i^j}\Bigr) \leq \sum_{i=1}^{k}\sum_{j=1}^{\alpha_i}\frac{1}{p_i^j}.
    \]
    Therefore, if $x$ is exceptional, then taking logarithms in \eqref{eq:except_defn} gives
    \[
        \sum_{\substack{q' \mid n-x\\q' \leq x}}\frac{1}{q'} \geq \log\log x - \log\log\log x - O(1),
    \]
    where the summation ranges over prime powers. Since the sum of reciprocal proper prime powers converges, we obtain \eqref{eq:rec_prime_div}, and \eqref{eq:rec_prime_not_div} then follows from Lemma \ref{lm:mertens_second}.
\end{proof}

Our final lemma says that the exceptional numbers grow very fast, and hence there cannot be many of them.

\begin{lemma}
\label{lm:except_bounds}
For any $x, x' \in E_n$ such that $x < x'$, we have $x \leq \exp\bigl(O((\log\log x')^3)\bigr)$. In particular,
\[
    |E_n| \leq \log^*n + O(1).
\]
\end{lemma}
\begin{proof}
    From \eqref{eq:rec_prime_div} and \eqref{eq:rec_prime_not_div}, the latter with $x'$ in place of $x$, we get
    \[
        \sum_{\substack{p\mid\gcd(n-x,n-x')\\p \leq x}}\frac{1}{p} \geq \sum_{\substack{p\mid n-x\\p \leq x}}\frac{1}{p} - \sum_{\substack{p\nmid n-x'\\p \leq x}}\frac{1}{p} \geq \log\log x - \log\log\log x - \log\log\log x' - O(1).
    \]
    Combining this with Corollary \ref{cor:rec_prime_div} and using
    \[
        \gcd(n-x,n-x') = \gcd(n-x, x'-x) \leq x'-x < x',
    \]
    we arrive at the conclusion that
    \[
        2\log\log\log x' \geq \log\log x - \log\log\log x - O(1).
    \]
    Upon rearranging and exponentiating twice, we obtain that $x \leq \exp\bigl(O((\log\log x')^3)\bigr)$. If we instead take logarithms, we get
    \begin{equation}\label{eq:logs}
        \log\log\log x \leq \log\log\log\log x' + O(1).
    \end{equation}
    The conclusion about the size of $E_n$ now follows by iteration. Indeed, let $E_n = \{x_0, x_1,\ldots,x_{r-1}\}$, where $x_0 > x_1 > \ldots > x_{r-1}$. Let $s \in \{0,1,\ldots,r-1\}$ be minimal such that $\log\log\log x_s < 2B+2$, where $B$ is the implied constant in \eqref{eq:logs}; set $s = r$ if such an index does not exist. Writing $D \vcentcolon= \exp(\exp(\exp(2B+2)))$, we have $\{x_s,\ldots,x_{r-1}\} \subseteq [0, D\bigr)$, so it follows that $r-s \leq D$. If $s = 0$, we are done, so we may assume that $s > 0$. We will show by induction on $j$ that $\log\log\log x_j \leq \log^{(j+3)}x_0 + 2B$ and $\log^{(j+3)}x_0 \geq 2$ for all $0 \leq j < s$. Note that the second statement follows from the first and the fact that $\log\log\log x_j \geq 2B+2$, so it remains to prove the first statement. If $j = 0$, this is clear, and if $j \geq 1$, then by \eqref{eq:logs} with $x = x_j$, $x' = x_{j-1}$ together with the induction hypothesis,
    \[
        \log\log\log x_j \leq \log\log\log\log x_{j-1} + B \leq \log(\log^{(j+2)}x_0+2B) + B.
    \]
    Hence, using the inequality $\log(1+t) \leq t$, we get
    \[
        \log\log\log x_j - \log^{(j+3)}x_0 \leq \log\Bigl(1+\frac{2B}{\log^{(j+2)}x_0}\Bigr) + B \leq \frac{2B}{2} + B = 2B,
    \]
    which completes the induction step. Finally, $\log^{(s+2)}x_0 \geq 2$ implies that $\log^*x_0 \geq s+3$, and hence $r \leq \log^*n + D$, as desired.
\end{proof}

We are now in a position to prove Theorem \ref{thm:except_collision}. This is more or less a matter of putting together all the ingredients we have developed so far. We will, however, require a further observation. Specifically, in the case when the Markov chain $(Z_j^{(n)})_{j\geq0}$ jumps straight away into the interval $[X_4,X_3)$, Landau's theorem allows us to restrict attention to orders not much larger than $n$. By Lemma \ref{lm:div_in_int_simple}, this in turn means that the first step is uniquely determined by the order.

\begin{proof}[Proof of Theorem \ref{thm:except_collision}]
    By the triangle inequality and the fact that $0 \leq \wt{p}_n^{(4)} \leq \wt{p}_n$,
    \begin{equation}\label{eq:triangle_ineq}
        \lVert \wt{p}_n^{(4)}1_{(0,M]}\rVert_q \leq \lVert \wt{p}_n1_{(0, M]}\rVert_q \leq \sum_{j=1}^{4}\lVert \wt{p}_n^{(j)}1_{(0, M]}\rVert_q.
    \end{equation}
    Now if $j \in [3]$, then by Lemma \ref{lm:cauchy_schwarz} with $N = M$, Lemma \ref{lm:collision_prob} and Lemma \ref{lm:interval_bounds}, we have
    \begin{equation}\label{eq:power_saving}
        \lVert\wt{p}_n^{(j)}1_{(0, M]}\rVert_q^q \ll X_j^{c'/2}\Biggl(\sum_{X_{j}\leq x<X_{j-1}}\frac{1}{x^{1+c'}} + |E_n\cap [X_{j},X_{j-1})|\cdot\frac{1}{X_{j}^{c'}}\Biggr) \ll X_{j}^{c'/2}\cdot\frac{1}{X_{j}^{c'}} = X_{j}^{-c'/2},
    \end{equation}
    where we used that, by Lemma \ref{lm:except_bounds}, there are at most two exceptional values of $x$ in the interval $[X_{j},X_{j-1})$. On the other hand, in the case $j = 4$, Landau's theorem \cite{landau} implies that for $x \in [X_4,X_3)$ we have 
    \[
        \ord(\pi_x) \leq \exp((1+o(1))\sqrt{x\log x}) \leq \exp\Bigl((\log n)^{(C_3+2)/4}\Bigr) = n^{o(1)}
    \]
    provided $x$ is large enough. Therefore, for sufficiently large $n$ we have
    \[
        \lcm(\ord(\pi_x),n-x) \leq \ord(\pi_x)\cdot n \leq n^{1+o(1)},
    \]
    so it follows from the definition \eqref{eq:recursion_range} that $\wt{p}_n^{(4)}(m) = 0$ for $m > n^{3/2}$ say. We conclude that 
    \begin{equation}\label{eq:bound_reduction}
        \lVert\wt{p}_n^{(4)}1_{(0, M]}\rVert_q = \lVert\wt{p}_n^{(4)}1_{(0, N]}\rVert_q
    \end{equation}
    with $N \vcentcolon= n^{3/2}$. Hence, taking $a = n-X_3$, $t = X_3$ and $r = 2$ in Lemma \ref{lm:div_in_int_simple}, for $m \leq N$ we have
    \[
        a^r/t^{r^2} = \frac{(n-X_3)^2}{X_3^4} \geq \frac{n^2}{4(\log n)^{4C_3}} > m,
    \]
    whence $\tau(m;n-X_3,n-X_4) \leq 1$. Therefore, by Lemma \ref{lm:cauchy_schwarz}, we have
    \[
        \lVert\wt{p}_n^{(4)}1_{(0, N]}\rVert_q^q = \sum_{0\leq x < X_3}\sum_{m \leq N}\P(\lcm(\ord(\pi_x),n-x) = m)^q.
    \]
    Combining this with Lemma \ref{lm:collision_prob}, we obtain
    \begin{equation}\label{eq:final_interval}
          \lVert\wt{p}_n^{(4)}1_{(0, N]}\rVert_q^q = \sum_{0\leq x < X}\sum_{m \leq N}\P(\lcm(\ord(\pi_x),n-x) = m)^q + R \leq |E_n| + O(1),
    \end{equation}
    where $R$ is as in the statement of the theorem. Finally, putting together \eqref{eq:triangle_ineq}, \eqref{eq:power_saving} and \eqref{eq:bound_reduction}, it follows that
    \[
        \lVert \wt{p}_n1_{(0, M]}\rVert_q = \lVert \wt{p}_n^{(4)}1_{(0, N]}\rVert_q + O((\log n)^{-\Omega(1)}).
    \]
    Taking $q$-th powers and and using \eqref{eq:final_interval} together with Lemma \ref{lm:except_bounds} and the mean value inequality, we obtain
    \[
        \lVert \wt{p}_n1_{(0, M]}\rVert_q^q = \sum_{0\leq x < X}\sum_{m \leq N}\P(\lcm(\ord(\pi_x),n-x) = m)^q + R + O((\log n)^{-\Omega(1)}).
    \]
    The conclusion now follows from Lemma \ref{lm:large_orders}.
\end{proof}

Finally, we deduce Theorems \ref{thm:collision_prob_upper_bound}, \ref{thm:collision_prob_explicit} and \ref{thm:collision_prob_average} from Theorem \ref{thm:except_collision}.

\begin{proof}[Proof of Theorem \ref{thm:collision_prob_upper_bound}]
    The conclusion is immediate from Theorem \ref{thm:except_collision} and Lemma \ref{lm:except_bounds}.
\end{proof}

\begin{proof}[Proof of Theorem \ref{thm:collision_prob_explicit}]
    Assuming $k$ is sufficiently large and taking $n = n_k$, ${X = \min(k,k+D)}$ in Theorem \ref{thm:except_collision}, it suffices to show that for $0 \leq x < X$, $\gcd(\ord(\pi_x),n-x) = \gcd(\ord(\pi_x),x-D)$, whereas for $x \geq X$, $x \not\in E_n$. The former is equivalent to proving that, for all primes $p$,
    \begin{equation}\label{eq:p_adic_equal}
        \min\bigl(\nu_p(\ord(\pi_x)), \nu_p(n-x)\bigr) = \min\bigl(\nu_p(\ord(\pi_x)), \nu_p(x-D)\bigr).
    \end{equation}
    Note that $\ord(\pi_x)$ divides $\lcm(1,\ldots,x)$, which in turn divides $n-D$. Therefore, we certainly have $\nu_p(\ord(\pi_x)) \leq \nu_p(n-D)$. In particular, if $x = D$ or $\nu_p(x-D) = \nu_p(n-D)$, then \eqref{eq:p_adic_equal} holds since both sides are equal to $\nu_p(\ord(\pi_x))$. Otherwise, since $|x-D| \leq k$, it follows that $x-D$ divides $n-D$, so $\nu_p(x-D) < \nu_p(n-D)$. Therefore, we have $\nu_p(n-x) = \nu_p(x-D)$ and so \eqref{eq:p_adic_equal} again holds.
    
    Turning to the latter statement, suppose to the contrary that there exists $x \geq X$ with $x \in E_n$. For any prime $p \mid n-x$ such that $p \leq k$ we have $p \mid n - D$ and hence $p \mid x-D$. Thus, by Corollary \ref{cor:rec_prime_div},
    \[
        \sum_{\substack{p \mid n-x\\p \leq k}}\frac{1}{p} \leq \sum_{p \mid x-D}\frac{1}{p} \leq \log\log\log x + O(1),
    \]
    whence Lemma \ref{lm:except_equiv} implies that
    \[
        \sum_{\substack{p \mid n-x\\k < p \leq x}}\frac{1}{p} \geq \log\log x - 2\log\log\log x - O(1).
    \]
    The left-hand side, however, is at most $\omega(n-x)/k$. We have $k \gg \log n/\log\log n$, and by Lemma \ref{lm:prime_divisor_bound}, $\omega(n-x) \ll \log n/\log\log n$. We thus conclude that
    \[
        \log\log x - 2\log\log\log x \ll 1,
    \]
    whence $x \ll 1$. This is a contradiction for $k$ large enough, so we are done.
\end{proof}

\begin{remark}
\label{rem:collision_prob_explicit}
The above proof in fact shows that the conclusion of Theorem \ref{thm:collision_prob_explicit} extends to any sequence $(n_k)_{k\geq1}$ satisfying $n_k \equiv D \pmod {\lcm(1,\ldots,k)}$ and tending to infinity not too fast, namely 
\[
    n_k \leq \exp\Bigl(\exp\Bigl(\Omega\Bigl(\Bigl(\frac{\log k}{\log\log k}\Bigr)^2\Bigr)\Bigr)\Bigr).
\]
\end{remark}

\begin{proof}[Proof of Theorem \ref{thm:collision_prob_average}]
    By Theorem \ref{thm:except_collision}, it suffices to show that the size of $E_n$ is $O(1)$ on average. To this end, note that
    \[
        \frac{1}{N}\sum_{n=1}^{N}|E_n| = \frac{1}{N}\sum_{0\leq x < N}|\{n \in [N] \mid x \in E_n\}|.
    \]
    Let us fix a large enough $x \in \{0,1,\ldots,N-1\}$ and estimate the proportion of $n \in [N]$ such that $x \in E_n$. Note that if $x \in E_n$, then by Corollary \ref{cor:reciprocal_divisors}, $L \vcentcolon= \lcm(1,\ldots,x)$ has a divisor $d \geq T$ such that $d \mid n-x$, where
    \[
        T \vcentcolon= \exp\Bigl(\exp\Bigl(\frac{c''\log x}{\log\log x}\Bigr)\Bigr)
    \]
    and $c'' > 0$ is an absolute constant. Fixing such a divisor $d$, the proportion of $n \in (x,N]$ such that $d \mid n-x$ is at most $1/d$. Hence, by the union bound, we have
    \[
        \frac{1}{N}|\{n \in [N] \mid x \in E_n\}| \leq \sum_{\substack{d \mid L\\d \geq T}}\frac{1}{d}.
    \]
    We estimate this sum using a variant of Rankin's trick in the style of \cite[\S3.10]{granville-smooth}. Let $\sigma \in (0,\frac{1}{2})$ be a small parameter, to be determined later. Then
    \[
        \sum_{\substack{d \mid L\\d \geq T}}\frac{1}{d} \leq \sum_{d \mid L}\frac{1}{d}\Bigl(\frac{d}{T}\Bigr)^{\sigma} = T^{-\sigma}\sum_{d\mid L}d^{\sigma-1} \leq T^{-\sigma}\exp\Bigl(O\Bigl(\sum_{p\leq x}p^{\sigma-1}\Bigr)\Bigr).
    \]
    A standard argument using partial summation and Lemma \ref{lm:chebyshev} shows that
    \[
        \sum_{p\leq x}p^{\sigma-1} \ll \frac{x^{\sigma}}{\sigma\log x}
    \]
    as long as say $\sigma \geq 2\log\log x/\log x$. Hence, choosing $\sigma = c''/\log\log x$, we obtain
    \[
        \frac{1}{N}|\{n \in [N] \mid x \in E_n\}| \leq T^{-\frac{c''}{\log\log x} + O(\frac{\log\log x}{\log x})} = \exp\Bigl(-\exp\Bigl(\Omega\Bigl(\frac{\log x}{\log\log x}\Bigr)\Bigr)\Bigr).
    \]
    This is a summable function of $x$, so the desired conclusion follows.
\end{proof}

\section{Shannon entropy}\label{sec:shannon_entropy}

In this section, we prove Theorem \ref{thm:shannon_entropy}. We start by recalling the expression for Shannon entropy:
\[
    H_1(\ord(\pi_n)) = \sum_{m\in\N}p_n(m)\log\Bigl(\frac{1}{p_n(m)}\Bigr).
\]
To prove the lower bound, we consider the contribution of the orders that are not unusually small and do not have an unusually large number of prime factors. Specifically, fixing $\varepsilon > 0$, these orders come from the following set:
\[
    E \vcentcolon= \Bigl\{m \in \N\ \Big|\ m \geq \exp\Bigl(\Bigl(\frac{1}{2}-\varepsilon\Bigr)\log^2n\Bigr),\ \omega(m) \leq 2\log n\log\log n\Bigr\}.
\]
By the main result of \cite{erdos-turan-groups-i},
\[
    \P\Bigl(\ord(\pi_n) \geq \exp\Bigl(\Bigl(\frac{1}{2}-\varepsilon\Bigr)\log^2n\Bigr)\Bigr) \to 1 \text{ as } n \to \infty.
\]
Furthermore, by \cite[pp.\ 152]{erdos-turan-groups-ii},
\[
    \P\bigl(\omega(\ord(\pi_n)) \leq 2\log n\log\log n\bigr) \to 1 \text{ as } n \to \infty.
\]
Therefore, if $n$ is sufficiently large, we have
\begin{equation}\label{eq:lower_tail_prob}
    \sum_{m\in E}p_n(m) = \P(\ord(\pi_n) \in E) \geq 1-\varepsilon.
\end{equation}
To conclude the proof of the lower bound on $H_1(\ord(\pi_n))$, it remains to place a suitable upper bound on $p_n(m)$ for $m \in E$. To this end, we note that, for a suitably large constant $B > 0$,
\[
    p_n(m) \leq \P(c(\pi_n) > B\log^2n) + \sum_{\ell \leq B\log^2n}\P(c(\pi_n) = \ell,\ m \mid \ord(\pi_n)).
\]
Taking $B > 0$ sufficiently large, Corollary \ref{cor:tail_bounds} implies that the first term on the right-hand side is at most $O(\exp(-\log^2n))$. On the other hand, Lemma \ref{lm:cycles_and_primes} implies that, for $\ell \leq B\log^2n$,
\[
    \P(c(\pi_n) = \ell,\ m \mid \ord(\pi_n)) \leq \frac{(B\log^2n)^{2\log n\log\log n}}{\exp\Bigl(\Bigl(\frac{1}{2}-\varepsilon\Bigr)\log^2n\Bigr)} \leq \exp\Bigl(-\Bigl(\frac{1}{2}-2\varepsilon\Bigr)\log^2n\Bigr),
\]
where the last inequality holds provided $n$ is sufficiently large. Putting these estimates together, we obtain that, for $n$ large enough and $m \in E$,
\[
    p_n(m) \leq \exp\Bigl(-\Bigl(\frac{1}{2}-3\varepsilon\Bigr)\log^2n\Bigr).
\]
For sufficiently small $\varepsilon > 0$, combining this with \eqref{eq:lower_tail_prob} yields
\[
    H_1(\ord(\pi_n)) \geq (1-\varepsilon) \cdot \Bigl(\frac{1}{2}-3\varepsilon\Bigr)\log^2n \geq \Bigl(\frac{1}{2}-4\varepsilon\Bigr)\log^2n,
\]
as desired.

We now turn to the proof of the upper bound. Again, let $\varepsilon > 0$ be arbitrary but fixed. Choose $B$ as a value of $\lambda_2$ in Corollary \ref{cor:tail_bounds} for which $Q(\lambda_2/2) \geq 2$ and define the parameters
\[
    U \vcentcolon= \exp\Bigl(\Bigl(\frac{1}{2}+\varepsilon\Bigr)\log^2n\Bigr), \quad V \vcentcolon= \exp(B\log^2n), \quad \rho \vcentcolon= \exp(-2B\log^2n).
\]
We consider the decomposition
\[
    H_1(\ord(\pi_n)) = \sum_{j=1}^{4}T_j, \quad \text{where} \quad T_j \vcentcolon= \sum_{m\in F_j}p_n(m)\log\Bigl(\frac{1}{p_n(m)}\Bigr),
\]
and the sets $F_j$ for $j \in [4]$ are defined as
\[
    F_1 \vcentcolon= (V,\infty), \quad F_2 \vcentcolon= (0, U],
\]
\[
    F_3 = (U,V] \cap \{m \in \N \mid p_n(m) \leq \rho\}, \quad F_4 = (U, V] \cap \{m\in \N \mid p_n(m) > \rho\}.
\]
We estimate the sums $T_j$ for $j \in [4]$ in turn. First, we have
\[
    \sum_{m\in F_1}p_n(m) = \P(\ord(\pi_n) > V) \leq \P(c(\pi_n) > B\log n) = O(n^{-2})
\]
and also
\[
    \sum_{m \in F_3}p_n(m) \leq |F_3| \cdot \rho \leq \exp(-B\log^2n).
\]
For any $m \in \N$ with $p_n(m) > 0$, we have the crude estimate $p_n(m) \geq 1/n!$, so we obtain
\[
    T_1 + T_3 \leq \log(n!)\sum_{m\in F_1 \cup F_3}p_n(m) \leq n\log n \cdot (O(n^{-2}) + \exp(-B\log^2n)) = o(1).
\]
Next, letting $G$ denote the event that $\ord(\pi_n) \in F_2$ and defining for $m \in \N$ the conditional probabilities 
\[
    p_n'(m) \vcentcolon= \P(\ord(\pi_n) = m \mid G),
\] 
the entropy of $\ord(\pi_n)$ conditional on $G$ can be expressed as
\[
    \E\Bigl[\log\Bigl(\frac{1}{p_n'(\ord(\pi_n))}\Bigr)\ \Big|\ G\Bigl] = \E\Bigl[\log\Bigl(\frac{1}{p_n(\ord(\pi_n))}\Bigr)\ \Big|\ G\Bigl] + \log\bigl(\P(G)\bigr).
\]
Since the distribution of $\ord(\pi_n)$ conditional on $G$ is supported on $(0,U]$, the left-hand side in the above display is at most $\log U$. By the Erdős--Turán law \cite{erdos-turan-groups-i}, the second term on the right-hand side tends to zero as $n \to \infty$. The first term on the right-hand side is at least $T_2$, so it follows that
\[
    T_2 \leq \Bigl(\frac{1}{2}+\varepsilon\Bigr)\log^2n + o(1).
\]
Finally, we have
\[
    T_4 \leq \log(1/\rho)\sum_{m\in F_4}p_n(m) = 2B\log^2n \cdot \P\bigl(\ord(\pi_n) > U\bigr) \leq \varepsilon\log^2n,
\]
where the last inequality holds by the Erdős--Turán law \cite{erdos-turan-groups-i} provided $n$ is sufficiently large. In conclusion, for large enough $n$ we certainly have
\[
    H_1(\ord(\pi_n)) \leq \Bigl(\frac{1}{2}+3\varepsilon\Bigr)\log^2n,
\]
as desired. This concludes the proof of Theorem \ref{thm:shannon_entropy}.

\bibliographystyle{plain}
\bibliography{references}

\end{document}